\documentclass[reqno,twoside,12pt]{amsart}
\usepackage{amssymb,amsfonts,amsthm,amsmath}
\usepackage{enumitem}
\usepackage[pagebackref=false]{hyperref}

\usepackage[hmargin=1in,vmargin=1in]{geometry}

\usepackage{cite}
\usepackage{color}

\newcommand{\ppp}[1]{\langle #1\rangle}
\newcommand{\sign}{{\rm sign}}
\newcommand{\ptset}{\omega}

\def\eqdef{\stackrel{\rm def}{=}}

\def\d{{\rm d}}
\def\ddt{\frac{\d}{\d t}}
\renewcommand{\tilde}{\widetilde}

\def\beq{\begin{equation}}
\def\eeq{\end{equation}}
\def\beqs{\begin{equation*}}
\def\eeqs{\end{equation*}}

\newcommand{\tnum}{\rm(\roman*)}
\newcommand{\rnum}{\rm(\alph*)}

\newtheorem{theorem}{Theorem}[section]
\newtheorem{lemma}[theorem]{Lemma}
\newtheorem{proposition}[theorem]{Proposition}

\newtheorem{definition}[theorem]{Definition}
\newtheorem{assumption}[theorem]{Assumption}

\theoremstyle{definition}
\newtheorem{remark}[theorem]{Remark}
\newtheorem{example}[theorem]{Example}

\newtheorem{notation}[theorem]{Notation}

\definecolor{darkred}{rgb}{.70,.12,.20}

\definecolor{darkgreen}{rgb}{.20,.52,.14}

\def\varep{\varepsilon}

\renewcommand{\geq}{\ge}

\newcommand{\R}{\ensuremath{\mathbb R}}

\newcommand{\N}{\ensuremath{\mathbb N}}
\newcommand{\Z}{\ensuremath{\mathbb Z}}

\newcommand{\bigo}{\mathcal O}

\numberwithin{equation}{section}

\title{\textbf{Infinite series asymptotic expansions for decaying solutions of dissipative differential equations with non-smooth nonlinearity}}
\date{\today}

\begin{document}

\author{Dat Cao$^1$}
\address{$^1$Department of Mathematics and Statistics,
Minnesota State University, Mankato\\
Mankato, MN 56001,  U. S. A.}
\email{dat.cao@mnsu.edu}

\author{Luan Hoang$^{2,*}$}
\address{$^2$Department of Mathematics and Statistics,
Texas Tech University\\
1108 Memorial Circle, Lubbock, TX 79409--1042, U. S. A.}
\email{luan.hoang@ttu.edu}
\thanks{$^*$Corresponding author.}

\author{Thinh Kieu$^3$}
\address{$^3$Department of Mathematics,
University of North Georgia, Gainesville Campus\\
3820 Mundy Mill Rd., Oakwood, GA 30566, U. S. A.}
\email{thinh.kieu@ung.edu}


\begin{abstract}
We study the precise asymptotic behavior of a non-trivial solution that converges to zero, as time tends to infinity, of dissipative systems of nonlinear ordinary differential equations. The nonlinear term of the equations may not possess a Taylor series expansion about the origin. This absence technically cripples previous proofs in establishing an asymptotic expansion, as an infinite series, for such a decaying solution. In the current paper, we overcome this limitation and obtain an infinite series asymptotic expansion, as time goes to infinity.
This series expansion provides large time approximations for the solution with the errors decaying exponentially at any given rates. The main idea is to shift the center of the Taylor expansions for the nonlinear term to a non-zero point. Such a point turns out to come from the non-trivial asymptotic behavior of the solution, which we prove by a new and simple method. Our result applies to different classes of non-linear equations that have not been dealt with previously.
\end{abstract}

\maketitle

\tableofcontents
 
 \pagestyle{myheadings}\markboth{D. Cao, L. Hoang and T. Kieu}
{Infinite Series Asymptotic Expansions for Decaying Solutions of Dissipative  Differential Equations}

\section{Introduction}\label{intro}

The Navier--Stokes equations (NSE) for a viscous, incompressible fluid in bounded or periodic domains with a potential body force can be written in the functional form as
\beq\label{NSE}
\frac{\d y}{\d t}+Ay+B(y,y)=0,
\eeq
where $A$ is the (linear) Stokes operator and $B$ is a bilinear form in appropriate functional spaces.

In \cite{FS84a}, Foias and Saut prove  that any regular solution $y(t)$ of \eqref{NSE} has a following asymptotic behavior, as $t\to\infty$,  
\beq\label{rawxi}
e^{\lambda t} y(t)\to \xi\text{ for some $\lambda>0$ and $\xi\ne 0$ with } A\xi=\lambda\xi.
\eeq

This result is extended later by Ghidaglia \cite{Ghidaglia1986a} to a more general class of parabolic inequalities. 
The proof in \cite{Ghidaglia1986a} uses the same Dirichlet quotient technique by Foias--Saut \cite{FS84a}.

In \cite{FS87}, Foias and Saut go further and prove the following  asymptotic expansion, as $t\to\infty$,
\beq\label{FSx}
y(t)\sim \sum_{n=1}^\infty q_n(t) e^{-\mu_n t},
\eeq
in all Sobolev spaces, where $q_n(t)$ are polynomials in $t$, valued in the space of smooth functions. See Definition \ref{Sexpand} below for the precise meaning of \eqref{FSx}. 
Their proof of \eqref{FSx} does not require the knowledge of \eqref{rawxi} and uses a completely different technique.

The expansion \eqref{FSx} is studied deeply in later work \cite{FS91,FHS1,FHOZ1,FHOZ2} concerning its convergence, associated normalization map, normal form, invariant nonlinear manifolds, relation with the Poincar\'e--Dulac theory, etc. It is applied to the  analysis of physics-oriented aspects of fluid flows \cite{FHN1,FHN2},
is established for the NSE in different contexts such as with the Coriolis force \cite{HTi1}, or with non-potential forces \cite{HM2,CaH1,CaH2}, is extended to dissipative wave equations in \cite{Shi2000}, is investigated for general ordinary differential equations (ODE) without forcing functions in \cite{Minea}, and with forcing functions in \cite{CaH3}. 
The considerations of ODE in \cite{Minea,CaH3} turns out to be fruitful, and prompts to the recently obtained asymptotic expansions for the Lagrangian trajectories of viscous, incompressible fluid flows in \cite{H4}.

\medskip
In the same spirit as \cite{Minea,CaH3},  we study, in this paper, the ODE systems in $\R^d$ of the form
\beq\label{vtoy}
\frac{\d y}{\d t} +Ay =F(y),\quad t>0,
\eeq
where $A$ is a $d\times d$ constant (real) matrix, and $F$ is a vector field on $\R^d$.

Our goal is to obtain the  asymptotic expansion \eqref{FSx}, as $t\to\infty$,  for any decaying solution $y(t)$ of \eqref{vtoy}, where $q_n(t)$'s are $\R^d$-valued polynomials in $t$. (For other approaches to the asymptotic analysis of the solutions, see discussions in Remark \ref{discuss} below.)

In all of the above cited papers, function $F$ in \eqref{vtoy} must be infinitely differentiable at the origin. 
It is due to the requirement that $F(y)$ can be approximated, up to arbitrary orders, near the limit of $y(t)$, i.e. the origin, by the polynomials that come from of the Taylor series of $F$. The current paper investigates the situation when this is not the case, and hence the results in \cite{Minea,Shi2000,CaH3} do not apply.

A standard and intuitive way to find expansion \eqref{FSx} is substituting it into equation \eqref{vtoy}, expanding both sides in $t$, and equating the coefficient functions of corresponding exponential terms. Because of the lack of the Taylor series of $F(y)$ about the origin, one does not know how to find the expansion in $t$  for $F(y(t))$ on the right-hand side of \eqref{eg1}. The task seems to be impossible. However, as will be proved later in this paper, we are still able to obtain the infinite series asymptotic expansion \eqref{FSx} for $y(t)$ in many cases. This is achieved by combining Foias--Saut's method in \cite{FS87} with the following new idea.
For illustrative purposes, we consider an example,
\beq\label{eg1}
\frac{\d y}{\d t}+Ay=F(y)=\frac{|y|^{1/3}y}{1+|y|^{1/4}}.
\eeq

First, we use the geometric series to approximate $F(y)$ by a series 
\beq \label{Fsum}
F(y)\sim \sum_{k=1}^\infty F_k(y)\text{ as }y\to 0,
\eeq 
where $F_k$'s are a positively homogeneous  functions of strictly increasing degrees $\beta_k\to\infty$. (In general cases, \eqref{Fsum} is a hypothesis.) See Definition \ref{phom} and Assumption \ref{assumpG} for details. After establishing the asymptotic approximation \eqref{rawxi} for some eigenvector $\xi$ of $A$, we approximate each $F_k$ by using its Taylor series about $\xi\ne 0$. Therefore, we can bypass the lack of the Taylor series of $F$ about $0$.
This, of course, is just a brief description and must be facilitated with capable techniques. 

The paper is organized as follows. In section \ref{bkgsec}, we set the assumptions for matrix $A$, establish basic properties and recall a crucial approximation lemma, Lemma \ref{exp-odelem}. 

In section \ref{secD}, we prove, for a more general equation \eqref{geneq} with a general structure \eqref{Fcond}, that any non-trivial, decaying solution has the first asymptotic approximation \eqref{rawxi}, see Theorem \ref{thmD2}. 
This result can be obtained by repeating Foias--Saut's proof in \cite[Proposition 3]{FS84a}, or applying \cite[Theorem 1.1]{Ghidaglia1986a}. However, our new proof provides an alternative method and, at least for the current setting, is shorter. 
See Remark \ref{compare} for comparisons between the proofs.

The paper's main result is in section \ref{secE}. In Theorem \ref{thm3}, we prove that any non-trivial, decaying solution of \eqref{vtoy} has an asymptotic expansion of the form \eqref{FSx}. In order to implement to general scheme of Foias--Saut's \cite{FS87}, we use the first approximation $e^{-\lambda t} \xi$ in \eqref{rawxi}. By the positive homogeneity of each function $F_k$ in \eqref{Fsum}, we can scale $y(t)$ by the factor  $e^{-\lambda t}$ and then shift the Taylor expansions of  $F_k$'s from center zero to center $\xi\ne 0$. Because of the above scaling and its effect during complicated iterations, the exponential rates must be shifted back, see the set $\tilde S$ in \eqref{deftilS}, and forth, see the set $S$ in \eqref{defMu1}, when being generated in Definition \ref{SS1}.

Although we focus on infinite series expansions in this paper, we consider, in the first part of section \ref{finapprox}, the case when the function $F(y)$ has only a finite sum approximation, see \eqref{errF3}.
We prove in Theorem \ref{thm51} that any decaying solution $y(t)$ has a corresponding finite sum approximation.
In the second part of section \ref{finapprox}, Theorem \ref{thm57} generalizes Theorems \ref{thm3} and \ref{thm51} by relaxing the conditions on functions $F$ and  $F_k$'s, in accordance with the knowledge of the eigenspaces of $A$. 

Section \ref{specific} is devoted to identifying some specific classes of functions $F$, see Theorems \ref{thm63}, \ref{thm65} and \ref{thm69}. Briefly speaking, these functions can be expanded in terms of power-like functions of the types  $x_i^{\gamma_i}$, $|x_i|^{\gamma_i}$, $|x_i|^{\gamma_i}\sign(x_i)$ for coordinates $x_i$'s of $x\in\R^d$,  or of type $\|x\|_p^\gamma$, or, more generally, $\|P(x)\|_p^\gamma$ with $\ell^p$-norms  $\|\cdot\|_p$, where $P$ is a homogeneous polynomial. 
Lastly, we compare, in Remark \ref{discuss}, our results with other asymptotic expansion theories for ODE, notably the one that has been developed by Bruno and collaborators, see \cite{BrunoBook2000,Bruno2004,Bruno2012,Bruno2018} and references therein.

\section{Notation, definitions and background}\label{bkgsec}

We will use the following notation throughout the paper. 

\begin{itemize}
 \item $\N=\{1,2,3,\ldots\}$ denotes the set of natural numbers, and $\Z_+=\N\cup\{0\}$.

 \item Denote $\R_*=\R\setminus\{0\}$, and, for $n\in\N$,   $\R_*^n=(\R_*)^n$ and $\R^n_0=\R^n\setminus\{0\}$.
  
 \item For any vector $x\in\R^n$, we denote by $|x|$ its Euclidean norm, and by $x^{(k)}$ the $k$-tuple $(x,\ldots,x)$ for $k\ge 1$, and  $x^{(0)}=1$. 
 
 \item For an $m\times n$ matrix $M$, its Euclidean norm in $\R^{mn}$ is denoted by $|M|$. 
 
 \item Let $f$ be an $\R^m$-valued function and $h$ be a non-negative function, both  are defined in a neighborhood of the origin in $\R^n$. 
 We write $f(x)=\bigo(h(x))$ as $x\to 0$, if there are positive numbers  $r$ and $C$ such that $|f(x)|\le Ch(x)$ for all $x\in\R^n$ with $|x|<r$.   
 
 \item Let $f:[T_0,\infty)\to \R^n$ and $h:[T_0,\infty)\to[0,\infty)$ for some $T_0\in \R$. We write 
 $$f(t)=\mathcal O(h(t)), \text{ implicitly meaning as $t\to\infty$,} $$
 if there exist numbers $T\ge T_0$ and $C>0$ such that $|f(t)|\le Ch(t)$ for all $t\ge T$.

 \item Let $T_0\in\R$, functions $f,g:[T_0,\infty)\to\R^n$, and $h:[T_0,\infty)\to[0,\infty)$. 
We will conveniently write  $f(t)=g(t)+\bigo(h(t))$ to indicate
  $f(t)-g(t)=\bigo(h(t))$.  
\end{itemize}

The type of asymptotic expansions at time infinity that is studied in this paper is the following.

\begin{definition}\label{Sexpand}
Let $(X,\|\cdot\|_X)$ be a normed space and  $(\alpha_n)_{n=1}^\infty$ be a  sequence of strictly increasing non-negative real numbers. 
A function $f:[T,\infty)\to X$, for some $T\ge 0$, is said to have an asymptotic expansion
\beq\label{ex1}
f(t)\sim\sum_{n=1}^\infty f_n(t)e^{-\alpha_n t} \quad\text{in } X,
\eeq
where each $f_n:\R\to X$ is a polynomial, if one has, for any $N\ge 1$, that
\beq\label{ex2}
\Big \|f(t)-\sum_{n=1}^N f_n(t)e^{-\alpha_n t}\Big \|_X=\bigo(e^{-(\alpha_N+\varep_N)t})\text{ for some $\varep_N>0$.}
\eeq
\end{definition}

One can see, e.g. \cite[Lemma 4.1]{CaH3}, that the polynomials $f_1,f_2,\ldots,f_N$ in \eqref{ex2} are unique. 

In the case $\alpha_n\to\infty$ as $n\to\infty$, the (infinite series) asymptotic expansion \eqref{ex1} provides exponentially precise approximations for $f(t)$, as $t\to\infty$.
More specifically, for any $\gamma>0$, the partial sum $\sum_{n=1}^N f_n(t)e^{-\alpha_n t}$ of the series, with sufficiently large $N$, approximates $f(t)$, as $t\to\infty$, with an error of order $\bigo(e^{-\gamma t})$.

Regarding the nonlinearity in \eqref{vtoy}, the function $F$ will be approximated near the origin by functions, not necessarily polynomials,  in the following class.

\begin{definition}\label{phom} 
Suppose $(X,\|\cdot\|_X)$ and $(Y,\|\cdot\|_Y)$ be two (real) normed spaces.

A function $F:X\to Y$ is positively homogeneous of degree $\beta\ge 0$ if
\beq\label{Fb}
F(tx)=t^\beta F(x)\text{ for any $x\in X$ and any $t>0$.}
\eeq

Define $\mathcal H_\beta(X,Y)$ to be the set of positively homogeneous  functions of order $\beta$ from $X$ to $Y$, and  
denote $\mathcal H_\beta(X)=\mathcal H_\beta(X,X)$.

For a function $F\in \mathcal H_\beta(X,Y)$,  define
\beqs
\|F\|_{\mathcal H_\beta}=\sup_{\|x\|_X=1} \|F(x)\|_Y=\sup_{x\ne 0} \frac{\|F(x)\|_Y}{\|x\|_X^\beta}.
\eeqs
\end{definition}

The following are immediate properties.
\begin{enumerate}[label=\rnum]
 \item If $F\in \mathcal H_\beta(X,Y)$ with $\beta>0$, then taking $x=0$ and $t=2$ in \eqref{Fb} gives 
\beq\label{Fzero}  F(0)=0.
\eeq 
If, in addition, $F$ is bounded on the unit sphere in $X$, then  
\beq\label{Fhbound}
\|F\|_{\mathcal H_\beta}\in[0,\infty) \text{ and } 
\|F(x)\|_Y\le \|F\|_{\mathcal H_\beta} \|x\|_X^{\beta} \quad \forall x\in X.
\eeq

\item  The zero function (from $X$ to $Y$) belongs to $\mathcal H_\beta(X,Y)$ for all $\beta\ge 0$, and a constant function (from $X$ to $Y$) belongs to $\mathcal H_0(X,Y)$.

\item\label{ps} Each $\mathcal H_\beta(X,Y)$, for $\beta\ge 0$, is a linear space.

\item\label{pm} If $F_1\in \mathcal H_{\beta_1}(X,\R)$ and $F_2\in \mathcal H_{\beta_2}(X,Y)$, then $F_1F_2\in \mathcal H_{\beta_1+\beta_2}(X,Y)$.

\item\label{pp} If $F:X\to Y$ is a homogeneous polynomial of degree $m\in\Z_+$, then $F\in \mathcal H_m(X,Y)$. 
\end{enumerate}

In \ref{pp} above and throughout the paper, a constant function, even when it is zero, is considered as a homogeneous polynomial of degree zero.

\medskip
The space $\mathcal H_\beta(X,Y)$ can contain much more complicated functions than homogeneous polynomials. 
For example, let $s\in \Z_+$, numbers $\nu_j$, for $1\le j\le s$, be positive, $P_j$, for $1\le j\le s$, be a homogeneous polynomial of degree $m_j\in\N$ from $X$ to a normed space $(Y_j,\|\cdot\|_{Y_j})$. Let $P_0:X\to Y$ be  homogeneous polynomial of degree $m_0\in\Z_+$. 
Consider function $F$ defined by  
\beq\label{heg}
F(x)=\|P_1(x)\|_{Y_1}^{\nu_1} \|P_2(x)\|_{Y_2}^{\nu_2} \ldots \|P_s(x)\|_{Y_s}^{\nu_s} P_0(x),\text{ for $x\in X$.}
\eeq

Then one has
\beq\label{simF}
F\in\mathcal H_\beta(X,Y), \text{ where } \beta=m_0+\sum_{j=1}^s m_j\nu_j.
\eeq 

Thanks to \eqref{simF} and property \ref{ps} above, any linear combination of functions of the form in \eqref{heg} with the same number $\beta$ also  belongs to $\mathcal H_\beta(X,Y)$.

\medskip
If $n,m,k\in \N$ and $\mathcal L$ is an $m$-linear mapping from $(\R^n)^m$ to $\R^k$, the norm of $\mathcal L$ is defined by
\beq\label{Lnorm}
\|\mathcal L\|=\max\{ |\mathcal L(x_1,x_2,\ldots,x_m)|:x_j\in\R^n,|x_j|=1,\text{ for } 1\le j\le m\}.
\eeq

It is known that the norm $\|\mathcal L\|$ belongs to  $[0,\infty)$, and one has 
\beq\label{multiL}
|\mathcal L(x_1,x_2,\ldots,x_m)|\le \|\mathcal L\|\cdot |x_1|\cdot |x_2|\ldots |x_m|
\quad \forall x_1,x_2,\ldots,x_m\in\R^n.
\eeq

In particular, when $m=1$, \eqref{Lnorm} yields the operator norm for any $k\times n$ matrix $\mathcal L$.

\medskip
Let the space's dimension $d\in \N$ be fixed throughout the paper.
Consider the ODE system \eqref{vtoy}.

\begin{assumption}\label{assumpA}
Hereafter, matrix $A$ is a diagonalizable with positive eigenvalues.
\end{assumption}
  
Thanks to Assumption \ref{assumpA}, the spectrum $\sigma(A)$ of matrix $A$ consists of eigenvalues $\Lambda_k$'s, for $1\le k\le d$,  which are positive and increasing in $k$.
Then there exists an invertible matrix $S$ such that
\beqs
A=S^{-1} A_0 S,
\text{ where } A_0={\rm diag}[\Lambda_1,\Lambda_2,\ldots,\Lambda_d].
\eeqs 

Denote the distinct eigenvalues of $A$ by $\lambda_j$'s that are strictly increasing in $j$, i.e., 
 \beqs
 0<\lambda_1=\Lambda_1<\lambda_2<\ldots<\lambda_{d_*}=\Lambda_d\quad \text{ with } 1\le d_*\le d.
 \eeqs
 
For $1\le k,\ell\le d$, let $E_{k\ell}$ be the elementary $d\times d$ matrix $(\delta_{ki}\delta_{\ell j})_{1\le i,j\le d}$, where $\delta_{ki}$ and $\delta_{\ell j}$ are the Kronecker delta symbols.

For $\lambda\in \sigma(A)$, define
 \beqs
\hat R_\lambda=\sum_{1\le i\le d,\Lambda_i=\lambda}E_{ii}\text{ and } R_\lambda=S^{-1}\hat R_\lambda S.
 \eeqs
 
Then one immediately has
\beq \label{Pc} I_d = \sum_{j=1}^{d_*} R_{\lambda_j},
\quad R_{\lambda_i}R_{\lambda_j}=\delta_{ij}R_{\lambda_j},
\quad  AR_{\lambda_j}=R_{\lambda_j} A=\lambda_j R_{\lambda_j},
\eeq 
and there exists $c_0\ge 1$ such that
\beq\label{Requiv}
c_0^{-1}|x|^2\le \sum_{j=1}^{d_*} |R_{\lambda_j}x|^2\le c_0|x|^2 \text{ for all } x\in\R^d.
\eeq

Below, we recall  a key approximation lemma for linear ODEs. It is Lemma 2.2 of \cite{CaH3}, which originates from Foias--Saut's work \cite{FS87}, and is based on the first formalized version \cite[Lemma 4.2]{HM2}.

\begin{lemma}[{\cite[Lemma 2.2]{CaH3}}]\label{exp-odelem}
Let $p(t)$ be an $\R^d$-valued polynomial and  $g:[T,\infty)\to \R^d$, for some  $T\in\R$, be a continuous function satisfying $|g(t)|=\bigo(e^{-\alpha t})$ for some $\alpha>0$.
Suppose $\lambda > 0$ and $y\in C([T,\infty),\R^d)$ is a solution of 
 \beqs
 y'(t)=-(A-\lambda I_d)y(t)+p(t)+g(t),\quad \text{for }t\in(T,\infty).
 \eeqs

If $\lambda>\lambda_1$, assume further that
\beq\label{ellams}
\lim_{t\to\infty} (e^{(\bar\lambda -\lambda )t}|y(t)|)=0,\text{ where } \bar\lambda=\max\{\lambda_j: 1\le j\le d_*,\lambda_j<\lambda\}.
\eeq 
 
Then there exists a unique $\R^d$-valued polynomial $q (t)$ such that 
\beq\label{qeq} q'(t)=-(A-\lambda I_d) q (t) +  p (t)  \text{ for }t\in\R,
\eeq  
and
 \beq\label{yqest}
 |y(t)-q(t)|=\bigo(e^{-\varep  t}) \text{ for some $\varepsilon>0$.}
 \eeq
 \end{lemma}

In fact, the polynomial $q(t)$ in Lemma \ref{exp-odelem} can be defined explicitly as follows.
We write, with the use of \eqref{Pc}, $q(t)=\sum_{j=1}^{d_*} R_{\lambda_j} q(t)$, where, for each $1\le j\le d_*$ and $t\in \R$, 
\beq\label{qdef}        
R_{\lambda_j} q(T+t)=
\begin{cases}
e^{-(\lambda_j-\lambda) t}\int_{0}^t e^{(\lambda_j-\lambda)\tau }R_{\lambda_j}p(T+\tau) d\tau&\text{if }\lambda_j>\lambda,\\
R_{\lambda_j}y(T) +\int_0^\infty  R_{\lambda_j}g(T+\tau)d\tau + \int_0^t R_{\lambda_j}p(T+\tau)d\tau &\text{if }\lambda_j=\lambda,\\
-e^{-(\lambda_j-\lambda) t}\int_t^\infty e^{(\lambda_j-\lambda)\tau }R_{\lambda_j}p(T+\tau) d\tau&\text{if }\lambda_j<\lambda. 
\end{cases}
\eeq

In the case $p(t)\equiv 0$, it follows \eqref{qdef} that $q(t)\equiv \xi$, which is a constant vector in $\R^d$. 
Then \eqref{qeq} and \eqref{yqest} read as
\beq\label{qxi}
(A-\lambda I_d)\xi=0\text{ and }|y(t)-\xi|=\bigo(e^{-\varep t}).
\eeq

\section{The first asymptotic approximation}\label{secD}

Consider the following ODE on $\R^d$, which is more general than \eqref{vtoy},
 \beq  \label{geneq}
\frac{\d y}{\d t} + Ay =F(t,y), \quad t>0.
 \eeq

 \begin{assumption}\label{assumpF}
Function $F$ mapping $(t,x)\in[0,\infty)\times \R^d$ to $F(t,x)\in\R^d$ is continuous in $[0,\infty)\times \R^d$, locally Lipschitz with respect to $x$ in $[0,\infty)\times \R^d$, and there exist positive numbers $c_*, \varep_*, \alpha$ such that 
\beq \label{Fcond}
 |F(t,x)| \le c_*|x|^{1+\alpha} \ \forall  t \ge 0,  \ \forall x \in \R^d  \text{ with } |x| \le \varep_*.
\eeq  
 \end{assumption}

It follows \eqref{Fcond} that $F(t,0)=0$ for all $t\ge 0$.
By the uniqueness/backward uniqueness of ODE system \eqref{geneq}, a solution $y(t)\in C^1([0,\infty))$ of \eqref{geneq} has the property
\beq \label{bwu}
y(0)=0\text{ if and only if } y(t)=0 \text{ for all }t\ge 0.
\eeq 

Thanks to Assumption \ref{assumpA} and \eqref{Fcond}, it is well-known that the trivial solution $y(t)\equiv 0$ of \eqref{vtoy} is asymptotically stable, see, for example, \cite[Theorem 1.1, Chapter 13]{CL55}. 

A solution  $y(t)\in C^1([0,\infty))$ of \eqref{geneq}  that satisfies $y(0)\ne 0$  and
\beq\label{ylim}
\lim_{t\to\infty} y(t)=0,
\eeq 
will be referred to as a \emph{non-trivial, decaying solution}. These solutions will be the focus of our study.

The following elementary result provides, for non-trivial, decaying solutions, a more precise upper bound, compared to \eqref{ylim}, and an additional lower bound.

\begin{proposition}\label{nzprop}
Let $y(t)$ be a non-trivial, decaying solution of \eqref{geneq}. Then there exists a number $C_1>0$ such that
\beq\label{yexp1}
|y(t)|\le C_1 e^{-\Lambda_1 t} \text{ for all $t\ge 0$.}
\eeq

Moreover, for any $\varep>0$, there exists a number $C_2=C_2(\varep)>0$ such that
\beq\label{yexp2}
|y(t)|\ge  C_2 e^{-(\Lambda_d+\varep) t} \text{ for all $t\ge 0$.}
\eeq 
\end{proposition}
\begin{proof}
 Set $Y(t)=(\sum_{j=1}^{d_*} |R_{\lambda_j}y(t)|^2)^{1/2}$. 
Applying $R_{\lambda_j}$ to equation \eqref{geneq}, taking dot product of the resulting equation with $R_{\lambda_j}y$, using the last property in \eqref{Pc}, and then summing over $j$, we obtain
 \beq\label{yraw}
 \frac12\ddt Y^2(t)= \frac12\ddt \sum_{j=1}^{d_*} |R_{\lambda_j}y|^2=-\sum_{j=1}^{d_*}\lambda_j |R_{\lambda_j}y|^2 +\sum_{j=1}^{d_*} R_{\lambda_j}F(t,y) \cdot R_{\lambda_j}y.
 \eeq 

 Note that 
\beq\label{LY}
\Lambda_1 \sum_{j=1}^{d_*} |R_{\lambda_j}y|^2\le \sum_{j=1}^{d_*}\lambda_j |R_{\lambda_j}y|^2\le \Lambda_d \sum_{j=1}^{d_*} |R_{\lambda_j}y|^2.
\eeq

Denote $C_0=\sum_{j=1}^{d_*} \|R_{\lambda_j}\|^2$.  
Let $\varep>0$ be arbitrary. By \eqref{ylim} and the asymptotic stability of the trivial solution, there exists $T_\varep\ge 0$ such that 
\beq\label{Te}
|y(t)|\le \varep_*\text{ and } C_0c_*c_0|y(t)|^\alpha\le \varep \ \forall t\ge T_\varep.
\eeq 

We have, for $t\ge T_\varep$,
 \beq\label{RF0}
\Big|\sum_{j=1}^{d_*} R_{\lambda_j}F(t,y) \cdot R_{\lambda_j}y \Big| \le \sum_{j=1}^{d_*} \|R_{\lambda_j}\|^2 |F(t,y)|\cdot |y|
\le C_0 c_*|y|^{2+\alpha}.
 \eeq

 Combining \eqref{RF0} with \eqref{Requiv} and \eqref{Te} gives
 \beq\label{RFy}
\Big|\sum_{j=1}^{d_*} R_{\lambda_j}F(t,y) \cdot R_{\lambda_j}y \Big|
\le C_0 c_* |y|^\alpha \cdot c_0 Y^2(t)
\le \varep Y^2(t)\ \forall t\ge T_\varep.
 \eeq 
 
\noindent \textit{Proof of \eqref{yexp1}.} By equation \eqref{yraw}, the first inequality in \eqref{LY}, and \eqref{RFy}, we have  
 \begin{align*}
\frac12\ddt Y^2&\le -(\Lambda_1-\varep)Y^2 \ \forall t\ge T_\varep.
 \end{align*}
 
 Thus, for $t\ge T_\varep$, 
 $$Y^2(t)\le Y^2(T_\varep)e^{-2(\Lambda_1-\varep)(t-T_\varep)}.$$
 
 Using this estimate and \eqref{RF0} in \eqref{yraw} gives, for $t>T_\varep$,
 \beqs
\frac12 \ddt Y^2
\le -\Lambda_1 Y^2 + C_0c_*(c_0^{1/2} Y)^{2+\alpha}
\le -\Lambda_1 Y^2 + C'_1 e^{-(2+\alpha)(\Lambda_1-\varep)(t-T_\varep)},
 \eeqs
hence,
 \begin{equation}\label{yraw2}
\ddt Y^2 \le -2\Lambda_1 Y^2 + 2C'_1 e^{-2\beta (t-T_\varep)},
 \end{equation}
 where $\beta=(1+\alpha/2)(\Lambda_1-\varep)$ and $C'_1$ is a positive number.  
 
 Choose $\varep$ sufficiently small so that $\beta>\Lambda_1$. Applying Gronwall's inequality to \eqref{yraw2}, for $t\ge T_\varep$, yields
 \begin{align*}
 Y^2(t)&\le e^{-2\Lambda_1 (t-T_\varep)}Y^2(T_\varep) +2C'_1\int_{T_\varep}^t e^{-2\Lambda_1(t-\tau)}e^{-2\beta(\tau-T_\varep)}\d\tau, \\
 \intertext{and, also by \eqref{Requiv},}
 |y(t)|^2&\le c_0Y^2(t)\le e^{-2\Lambda_1 (t-T_\varep)}c_0\Big(Y^2(T_\varep)+\frac{C'_1}{\beta-\Lambda_1}\Big).
 \end{align*}
 
Therefore, we obtain the inequality in \eqref{yexp1} for some constant $C_1>0$, but only for all $t\ge T_\varep$. 
Combining this with the boundedness of $|y(t)|$ on $[0,T_\varep]$, we then obtain estimate \eqref{yexp1} for all $t\ge 0$ with an adjusted constant $C_1>0$.

\medskip
\noindent \textit{Proof of \eqref{yexp2}.} By equation \eqref{yraw}, the second inequality in \eqref{LY}, and \eqref{RFy}, we have
 \beqs
 \frac12\ddt Y^2
 \ge -\Lambda_d Y^2 -\varep Y^2=- (\Lambda_d + \varep) Y^2 \quad  \forall t>T_\varep.
 \eeqs
 Hence,
 \beqs
 Y^2(t)\ge Y^2(T_\varep) e^{-2(\Lambda_d + \varep)(t-T_\varep)}\quad \forall t\ge T_\varep.
 \eeqs
 
By the virtue of \eqref{bwu}, $|y(t)|> 0$ for all $t\ge 0$.  
It follows that  
\beq \label{ylowexp}
 |y(t)|^2\ge c_0^{-1}Y^2(t)\ge c_0^{-2}|y(T_\varep)|^2 e^{-2(\Lambda_d + \varep)(t-T_\varep)}
 =C'_2 e^{-2(\Lambda_d + \varep)t} \quad \forall t\ge T_\varep,
 \eeq 
where $C'_2>0$. Since $y\in C([0,T_\varep],\R^d)$ and $|y(t)| > 0$ on $[0,T_\varep]$, one has $|y(t)|$ it is bounded below by a positive constant on $[0,T_\varep]$. Combining this fact with estimate \eqref{ylowexp} for $t\ge T_\varep$, we obtain the all-time estimate \eqref{yexp2}.
\end{proof}

The lower bound \eqref{yexp2} in Proposition \ref{nzprop} can be derived by using results for abstract problems in infinite dimensional spaces such as \cite[Theorems 1.1 and 1.2]{Ghidaglia1986b}, see also \cite{CohenLees}. Nonetheless, the proof above is included for being self-contained and simple.  

As discussed in the Introduction, the next theorem either follows the proof of \cite[Proposition 3]{FS84a}, or is a consequence of \cite[Theorem 1.1]{Ghidaglia1986a}. However, the proof presented below uses a new method, which may be useful in other problems.

\begin{theorem}\label{thmD2}
Let $y(t)$ be a non-trivial, decaying solution of \eqref{geneq}. Then there exist an eigenvalue $\lambda_*$ of $A$ and a corresponding eigenvector $\xi_*$ such that
\beq\label{firstrate}
|y(t)-e^{-\lambda_* t} \xi_*|=\bigo(e^{-(\lambda_*+\delta) t})\text{ for some }\delta>0.
\eeq
\end{theorem}
\begin{proof}
Define the set
\beq\label{Srate}
S'=\left\{ \sum_{j=1}^n \lambda'_j +m \alpha \lambda_1:\text{ for any numbers } n\in\N, \lambda'_j\in\sigma(A),  0\le m\in\Z\right\}.
\eeq

The set $S'$ can be arranged as a strictly increasing sequence $\{\nu_n\}_{n=1}^\infty$.
Note that $\nu_1=\lambda_1$ and $\nu_n\to\infty$ as $n\to\infty$.
For any $n\in\N$, one has $\nu_n+\alpha \lambda_1 >\nu_n$ and $\nu_n+\alpha \lambda_1 \in S'$. Hence, by the strict increase of $\nu_n$'s, we have
\beq\label{nunu}
\nu_n+\alpha \lambda_1 \ge \nu_{n+1}.
\eeq

\textit{Step 1.} First, by Proposition \ref{nzprop}, $|y(t)|\le Ce^{-\nu_1 t}$. Let $w_0(t)=e^{\nu_1 t} y(t)$. Then $w_0(t)$ satisfies
\beq\label{wz}
w_0'(t)+(A-\nu_1 I_d)w_0(t)=g_1(t)\eqdef e^{\nu_1 t}F(t,y(t)). 
\eeq
We estimate the right-hand side 
\beq\label{g1}
|g_1(t)|\le C e^{\nu_1 t}|y(t)|^{1+\alpha} \le C e^{\nu_1 t}e^{-\nu_1(1+\alpha) t}= \bigo(e^{-\alpha \nu_1 t}).
\eeq

By equation \eqref{wz} and estimate \eqref{g1}, we can apply Lemma \ref{exp-odelem} to $y(t)=w_0(t)$ and $p(t)\equiv 0$.
Then, by and \eqref{qxi}, there exists a vector $\xi_1\in\R^d$ and a number $\varep_1>0$ such that
\begin{align}\label{zx2}
A\xi_1&=\nu_1\xi_1,\\
\label{zx1}
|w_0(t)-\xi_1|&=\bigo(e^{-\varep t}), \text{ that is }, |y(t)-e^{-\nu_1 t}\xi_1|=\bigo(e^{-(\nu_1+\varep_1) t}).
\end{align}

\textit{Step 2.} Set 
$M=\{n\in\N:|y(t)|=\bigo(e^{-(\nu_n+\delta)t}) \text{ for some }\delta>0\}$.

Suppose $n\in M$. Let $w_n(t)=e^{\nu_{n+1}t} y(t)$. Then
\beq\label{wn}
w_n'(t)+(A-\nu_{n+1} I_d)w_n(t)=g_{n+1}(t)\eqdef e^{\nu_{n+1} t}F(t,y(t)).
\eeq

To estimate the last term, we note from \eqref{nunu} that $\nu_n(1+\alpha)\ge \nu_n+\lambda_1\alpha \ge \nu_{n+1}$.
Then, for large $t$,
\beq \label{gn1}
|g_{n+1}(t)|\le Ce^{\nu_{n+1} t} |y(t)|^{1+\alpha}\le Ce^{\nu_{n+1} t} e^{-(\nu_n+\delta)(1+\alpha)t}|=\bigo(e^{-\delta(1+\alpha) t}).
\eeq

By \eqref{wn} and \eqref{gn1}, we, again, can apply Lemma \ref{exp-odelem} to $y(t)=w_n(t)$ and $p(t)\equiv 0$. Then, by \eqref{qxi},  there exists a vector $\xi_{n+1}\in\R^d$ and a number $\varep>0$ such that
\begin{align*}
A\xi_{n+1}&=\nu_{n+1}\xi_{n+1},\\
|w_n(t)-\xi_{n+1}|&=\bigo(e^{-\varep t}),\text{ that is },|y(t)-e^{-\nu_{n+1}t}\xi_{n+1}|=\bigo(e^{-(\nu_{n+1}+\varep) t}) . 
\end{align*}

\textit{Step 3.} If the vector $\xi_1$ in Step 1 is not zero, then, thanks to \eqref{zx1} and \eqref{zx2}, the theorem is proved with $\lambda_*=\lambda_1$ and $\xi_*=\xi_1$.

Now, consider $\xi_1= 0$.  By \eqref{zx1} with $\xi_1=0$, one has $1\in M$, hence $M$ is a non-empty subset of  $\N$. By \eqref{yexp2} and the fact $\nu_n\to\infty$, the set $M$ must be finite. 
Let $k$ be the maximum number of $M$, and $n_0=k+1$. By the result in Step 2 applied to $n=k$, there exist $\xi_{n_0}\in\R^d$ and $\varep>0$  such that
\begin{align}
\label{Ano}
A\xi_{n_0}&=\nu_{n_0} \xi_{n_0},\\
\label{yno}
|y(t)-e^{-\nu_{n_0} t}\xi_{n_0}|&=\bigo(e^{-(\nu_{n_0} +\varep) t}).
\end{align}

If $\xi_{n_0}=0$, then \eqref{yno} implies $n_0\in M$, which is a contradiction.
Thus, $\xi_{n_0}\ne 0$, which, together with \eqref{Ano},   implies  $\lambda_*=\nu_{n_0}$ is an eigenvalue and $\xi_*=\xi_{n_0}$ is a corresponding eigenvector of $A$. Then, estimate \eqref{firstrate} follows \eqref{yno}.
\end{proof}

\begin{remark}\label{compare}
We compare the above proof of Theorem \ref{thmD2} with Foias--Saut's proof in \cite{FS84a}. We recall from \cite{FS84a} that the Dirichlet quotient $Ay(t)\cdot y(t)/|y(t)|^2$ is proved to converge, as $t\to\infty$, to an eigenvalue $\lambda_*$ of $A$ first,
and then, based on this, the two limits $e^{\lambda_* t}R_{\lambda_*} y(t)\to \xi_*\ne 0$ and $e^{\lambda_* t}(I_d-R_{\lambda_*}) y(t)\to 0$ are established. This original proof is rather lengthy and requires delicate analysis of the asymptotic behavior of $y(t)/|y(t)|$, see \cite[Proposition 1]{FS84a}.
We, instead, do not use the Dirichlet quotient to determine the exponential rate, but create the set $S'$ of possible rates, see \eqref{Srate}, and find the first $\lambda_*\in S'$ such that  $e^{\lambda_* t}|y(t)|$ does not decay exponentially. 
Then, by the virtue of approximation lemma \ref{exp-odelem},  estimate \eqref{firstrate}  is established without analyzing $y(t)/|y(t)|$.  This idea, in fact, is inspired by Foias--Saut's proof in \cite{FS87} of the asymptotic expansion \eqref{FSx}. However, we restrict it solely to the problem of first asymptotic approximation, and hence make it significantly simpler.
\end{remark}

\section{The series expansion}\label{secE}

In this section, we focus on obtaining the asymptotic expansion, as $t\to\infty$, for solutions of equation \eqref{vtoy}.
Regarding the equation's nonlinearity, we assume the following.
 
\begin{assumption}\label{assumpG}  The mapping $F:\R^d\to\R^d$ has the the following properties. 
\begin{enumerate}[label=\tnum]
 \item \label{FL} $F$ is locally Lipschitz on $\R^d$ and $F(0)=0$.
 \item\label{GG} Either \ref{h1} or \ref{h2} below is satisfied.

    \begin{enumerate}[label={\rm (H\arabic*)}]
        \item\label{h1} There exist numbers $\beta_k$'s, for $k\in\N$, which belong to $(1,\infty)$ and  increase strictly to infinity, and functions $F_k\in \mathcal H_{\beta_k}(\R^d)\cap C^\infty(\R^d_0)$, for $k\in\N$,  such that it holds, for any $N\in\N$, that  
        \beq\label{errF}
            \left|F(x)-\sum_{k=1}^N F_k(x)\right|=\bigo(|x|^{\beta})\text{ as $x\to 0$, for some $\beta>\beta_N$.}
        \eeq

        \item\label{h2} There exist $N_*\in\N$, strictly increasing  numbers $\beta_k$'s in $(1,\infty)$, 
        and functions $F_k\in \mathcal H_{\beta_k}(\R^d)\cap C^\infty(\R^d_0)$, for $k=1,2,\ldots,N_*$, such that  
        \beq\label{errF2}
            \left|F(x)-\sum_{k=1}^{N_*} F_k(x)\right|=\bigo(|x|^{\beta})\text{ as  $x\to 0$, for all $\beta>\beta_{N_*}$.}
        \eeq
    \end{enumerate}
\end{enumerate}
\end{assumption}

In Assumption \ref{assumpG}\ref{GG}, we conveniently write case \ref{h1} as
\beq\label{Gex}
F(x)\sim \sum_{k=1}^\infty F_k(x),
\eeq
and case \ref{h2} as
\beq\label{Gef}
F(x)\sim \sum_{k=1}^{N_*} F_k(x).
\eeq

The following remarks on Assumption \ref{assumpG} are in order.
\begin{enumerate}[label=\rnum]
 \item Applying \eqref{Fzero} and \eqref{Fhbound} to each  function $F_k$, one has 
$$F_k(0)=0,\quad \|F_k\|_{\mathcal H_{\beta_k}}<\infty,\text{ and } |F_k(x)|\le \|F_k\|_{\mathcal H_{\beta_k}}|x|^{\beta_k} \text{ for all $x\in\R^d$.}$$
Hence, \eqref{errF} indicates that the remainder $F(x)-\sum_{k=1}^N F_k(x)$ between $F(x)$ and its  approximate sum $\sum_{k=1}^N F_k(x)$ is small, as $x\to 0$, of a higher order (of $|x|$) than that in the approximate sum $\sum_{k=1}^N F_k(x)$. 

 \item With functions $F_k$'s as in \ref{h2} of Assumption \ref{assumpG}, if $F(x)=\sum_{k=1}^{N_*} F_k(x)$, then $F$ satisfies \eqref{Gef}.
For the relation between \eqref{Gex} and \eqref{Gef}, see Remark \ref{rmk44} below.

 \item By the remark \ref{pp} after Definition \ref{phom},  if $F$ is a $C^\infty$-vector field on the entire space $\R^d$ with $F(0)=0$ and $F'(0)=0$, then $F$ satisfies Assumption \ref{assumpG} with the right-hand side of \eqref{Gex} is simply the Taylor expansion of $F(x)$ about the origin.  
 
 \item Note that we do not require the convergence of the formal series on the right-hand side of \eqref{Gex}.
Even when the convergence occurs, the limit is not necessarily the function $F$. For instance, if $h:\R^d\to\R^d$ satisfies $|x|^{-\alpha} h(x)\to 0$ as $x\to 0$ for all $\alpha>0$, then $F$ and $F+h$ have the same expansion \eqref{Gex}.

 \item The class of functions $F$'s that satisfy Assumption \ref{assumpG} contains much more than smooth vector fields, see section \ref{specific} below.
\end{enumerate}

By Assumption \ref{assumpG}, for each $N\in\N$ in case of \eqref{Gex}, or $ N\in\N\cap[1, N_*]$ in case of \eqref{Gef}, there is $\varep_N>0$ such that
\beq\label{Ner}
\Big|F(x)-\sum_{k=1}^N F_k(x)\Big|=\bigo(|x|^{\beta_N+\varep_N})\text{ as } x\to 0.
\eeq

Note from \eqref{Ner} with $N=1$ that, as $x\to 0$, 
\beqs
|F(x)|\le |F_1(x)|+|F(x)-F_1(x)|\le \|F_1\|_{\mathcal H_{\beta_1}}|x|^{\beta_1}+\bigo(|x|^{\beta_1+\varep_1})=\bigo(|x|^{\beta_1}).
\eeqs
Thus, there exist numbers $c_*,\varep_*>0$ such that 
\beq \label{Gyy}
|F(x)|\le c_*|x|^{\beta_1} \quad\forall x\in\R^d \text{ with } |x|<\varep_*.
\eeq

\medskip
By property \eqref{Gyy} and Assumption \ref{assumpG}, function $F$ satisfies conditions in Assumption \ref{assumpF}. Therefore, the facts about trivial and non-trivial solutions in section \ref{secD} still applies to equation \eqref{vtoy}, and Theorem \ref{thmD2} holds true for solutions of  \eqref{vtoy}.

Hereafter, $y(t)$ is a non-trivial, decaying solution of \eqref{vtoy}. 

Let eigenvalue $\lambda_*=\lambda_{n_0}$ and its corresponding eigenvector $\xi_*$ be as in Theorem \ref{thmD2}. It follow \eqref{firstrate} that 
\beq\label{vraw}
 |y(t)|= \bigo( e^{-\lambda_* t} ). 
\eeq

To describe the exponential rates in a possible asymptotic expansion of solution $y(t)$ we use the following sets $\tilde S$ and $S$.

\begin{definition}\label{SS1}

We define a set $\tilde S\subset [0,\infty)$ as follows.

In the case of \eqref{Gex}, let $\alpha_k=\beta_k-1>0$ for $k\in\N$, and
\beq\label{deftilS}
\begin{aligned}
\tilde S=\Big\{& \sum_{k=n_0}^{d_*} m_k(\lambda_k-\lambda_*)+\sum_{j=1}^\infty z_j \alpha_j \lambda_*:m_k, z_j\in \Z_+,\\
& \text{ with $z_j> 0$ for only finitely many $j$'s} \Big\}.
\end{aligned}
\eeq

In the case of \eqref{Gef}, let $\alpha_k=\beta_k-1>0$ for $k=1,2,\ldots,N_*$, and 
\beq\label{tilfS}
\tilde S=\Big\{\sum_{k=n_0}^{d_*} m_k(\lambda_k-\lambda_*)+\sum_{j=1}^{N_*} z_j \alpha_j \lambda_*:m_k, z_j\in \Z_+\Big\}.
\eeq

In both cases,  the set $\tilde S$ has countably, infinitely many elements. 
Arrange $\tilde S$ as a sequence $(\tilde \mu_n)_{n=1}^\infty$ of non-negative and strictly increasing numbers.
Set 
\beq\label{defMu1}
\mu_n=\tilde \mu_n+\lambda_* \text{ for $n\in\N$, and define } S=\{\mu_n:n\in\N\}.
\eeq
\end{definition}

\medskip
The set $\tilde S$ has the following elementary properties.
\begin{enumerate}[label=\rnum]

\item For $n_0\le \ell\le d_*$, by choose $m_k=\delta_{k\ell}$, and $z_j=0$ for all $j$ in \eqref{deftilS} or \eqref{tilfS}, we have $\lambda_\ell-\lambda_*\in\tilde S$. Hence, 
\beq \label{ldS}  \lambda_\ell\in S\text{ for all }\ell=n_0,n_0+1,\ldots,d_*.
\eeq

\item Clearly, $\tilde \mu_1=0$ and $\mu_1=\lambda_*$. The numbers $\mu_n$'s are positive and strictly increasing. Also,
\beq\label{mulim}
\tilde\mu_n\to\infty \text{ and }\mu_n\to\infty \text{ as } n\to\infty.
\eeq

\item For all $x,y\in \tilde S$ and $k\in\N$, one has
\beq\label{Sprop} x+y,\ x+\alpha_k\lambda_*\in\tilde S.\eeq 
As a consequence of \eqref{Sprop}, one has  
\beq\label{muN}
\tilde \mu_n + \alpha_k \lambda_* \ge \tilde \mu_{n+1} \text{ for all } n,k.
\eeq
\end{enumerate}

Let $r\in \N$ and $s\in\Z_+$.  Since $F_r$ is a $C^\infty$-function in a neighborhood of $\xi_*\ne 0$, we have the following Taylor's expansion, for any $h\in\R^d$,
\beq\label{Taylor}
F_r(\xi_*+h)=\sum_{m=0}^s \frac1{m!}D^mF_r(\xi_*)h^{(m)}+g_{r,s}(h),
\eeq
where $D^m F_r(\xi_*) $ is the $m$-th order derivative of $F_r$ at  $\xi_*$, and
\beq\label{grs}
g_{r,s}(h)=\bigo(|h|^{s+1})\text{ as } h\to 0.
\eeq

For $m\ge 0$, denote
\beq\label{Frm}
\mathcal F_{r,m}=\frac1{m!}D^m F_r(\xi_*).
\eeq

When $m=0$, \eqref{Frm} reads as $\mathcal F_{r,0}=F_r(\xi_*)$. When $m\ge 1$, $\mathcal F_{r,m}$ is an $m$-linear mapping from $(\R^d)^m$ to $\R^d$.

By \eqref{multiL}, one has, for any $r,m\ge 1$, and $y_1,y_2,\ldots,y_m\in\R^d$, that
\beq\label{multineq}
|\mathcal F_{r,m}(y_1,y_2,\ldots,y_m)|\le \|\mathcal F_{r,m}\|\cdot |y_1|\cdot |y_2|\cdots |y_m|.
\eeq

For our convenience, we write inequality \eqref{multineq} even when $m=0$ with $\|\mathcal F_{r,0}\|\eqdef |F_r(\xi_*)|$.

Our main result is the following theorem.

\begin{theorem}\label{thm3}
There exist polynomials $q_n$: $\R\to \R^d$ such that  $y(t)$ has an asymptotic expansion, in the sense of Definition \ref{Sexpand}, 
\beq\label{vx}
y(t)\sim \sum_{n=1}^\infty q_n(t)e^{-\mu_n t} \text { in }\R^d,
\eeq
where $\mu_n$'s are defined in Definition \ref{SS1}, and $q_n(t)$ satisfies, for any $n\ge 1$,
 \beq\label{qneq}
q_{n}'+(A-\mu_{n}I_d)q_{n} 
=\mathcal J_n
\eqdef \sum_{\substack{r\ge1,m \ge 0,k_1,k_2,\ldots,k_{m} \ge 2,\\  
\sum_{j=1}^{m} \tilde\mu_{k_j}+\alpha_r\lambda_*=\tilde\mu_n}} \mathcal F_{r,m}(q_{k_1},q_{k_2},\ldots,q_{k_{m}})
\text{ in $\R$.}
\eeq
\end{theorem}

We clarify the notation in Theorem \ref{thm3}.
\begin{enumerate}[label=\rnum]
\item In case of assumption \eqref{Gex}, the index $r$ in $\mathcal J_n$ is taken over the whole set $\N$.
In case of assumption \eqref{Gef}, the index $r$ in $\mathcal J_n$ is restricted to $1,2,\ldots,N_*$, thus, we explicitly have 
\beq\label{Jnf}
\mathcal J_n=\sum_{r=1}^{N_*} \sum_{\substack{m \ge 0,k_1,k_2,\ldots,k_{m} \ge 2,\\  
\sum_{j=1}^{m} \tilde\mu_{k_j}+\alpha_r\lambda_*=\tilde\mu_n}} \mathcal F_{r,m}(q_{k_1},q_{k_2},\ldots,q_{k_{m}}). 
\eeq

\item When $m=0$, the terms $q_{k_j}$'s in $\mathcal J_n$ are not needed, see the explanation after \eqref{Frm}, hence the condition $k_j\ge 2$ is ignored, and the corresponding terms in $\mathcal J_n$ becomes 
\beq\label{mzero} 
\sum F_r(\xi_*) \text{ for } \alpha_r \lambda_*=\tilde\mu_n,\text{ that is, } \beta_r \lambda_*=\mu_n.
\eeq
Thus, we rewrite \eqref{qneq} more explicitly, by considering $m=0$ and $m\ge 1$ for $\mathcal J_n$, as
\beq\label{qmzero}
q_{n}'+(A-\mu_{n}I_d)q_{n} 
= \sum_{r\ge 1,\alpha_r\lambda_*=\tilde\mu_n} F_r(\xi_*)+ \sum_{\substack{r\ge 1, m \ge 1,k_1,k_2,\ldots,k_{m} \ge 2,\\  
\sum_{j=1}^{m} \tilde\mu_{k_j}+\alpha_r\lambda_*=\tilde\mu_n}} \mathcal F_{r,m}(q_{k_1},q_{k_2},\ldots,q_{k_{m}}). 
\eeq

Note, in \eqref{mzero}, that such an index $r$ may or may not exists. In the latter case, the term is understood to be zero. In the former case, $r$ is uniquely determined and we have only one term.

\item  When $n=1$, we have $\tilde\mu_1=0$, and there are no indices satisfying constraints for the sum in $\mathcal J_1$. Hence $\mathcal J_1=0$, and    \eqref{qneq} becomes
\beq \label{q1Eq}
 q_1' + (A-\mu_1 I_d)q_1 =0,
\eeq

\item  Consider $n=2$. If $m\ge 1$, then, for the second sum on the right-hand side of \eqref{qmzero}, one has  at least $\tilde \mu_{k_1}\ge \tilde\mu_2$. Hence 
$\tilde\mu_{k_j}+\alpha_r\lambda_*> \mu_{k_1} \ge \tilde\mu_2.$
Therefore, the last condition for the indices in the second sum on the right-hand side of \eqref{qmzero} is not met.
Thus, \eqref{qmzero} becomes 
\beqs
 q_2' + (A-\mu_2 I_d)q_2 =\mathcal J_2
 =\sum_{r\ge 1,\alpha_r\lambda_*=\tilde\mu_2} F_r(\xi_*)=\sum_{r\ge 1,\beta_r\lambda_*=\mu_2} F_r(\xi_*).
\eeqs

\item\label{checkfin}  We verify that the sum in $\mathcal J_n$  is a finite sum.

Let $n\ge 2$. Firstly, the indices in the sum of $\mathcal J_n$ satisfy
\beqs 
\tilde\mu_n=\sum_{j=1}^{m} \tilde \mu_{k_j}+\alpha_r\lambda_*\ge \alpha_r\lambda_*=\alpha_r \mu_1.
\eeqs
Then 
\beq\label{d1}
\alpha_r\le \tilde\mu_n/\mu_1.
\eeq 

Secondly, for $m\ge 1$, one has
\beqs
\tilde\mu_n=\sum_{j=1}^m \tilde\mu_{k_j}+\alpha_r\lambda_*>\sum_{j=1}^m \mu_{k_j} \ge m\tilde\mu_2,
\eeqs
which yields
\beq\label{d2}
m<\tilde\mu_n/\tilde\mu_2.
\eeq
Note that condition \eqref{d3} is not met for $n=2$ and $m\ge 1$.

Thirdly, $\tilde\mu_n=\sum_{j=1}^m \tilde\mu_{k_j}+\alpha_r\lambda_*> \tilde\mu_{k_j}$, which yields
\beq\label{d3}
k_j < n.
\eeq 
Hence, the terms $q_{k_j}$'s in \eqref{qneq} come from previous steps.     

By \eqref{d1}, \eqref{d2}, \eqref{d3}, the sum in $\mathcal J_n$ is over only finitely many $r$'s, $m$'s and $k_j$'s. 

\item For $n\ge 2$, suppose $r^*,m^*,k^*$ are non-negative integers such that 
\beq\label{sumcond} 
\alpha_{r^*}\ge \tilde \mu_n/\mu_1,\  m^*\ge \tilde\mu_n/\tilde\mu_2,\ k^*\ge n-1.
\eeq 
Then $\mathcal J_n$ can be equivalently written as 
\beq\label{finitesum}
\mathcal J_n=\sum_{r=1}^{r^*} \sum_{m=0}^{m^*} \sum_{\substack{2\le k_1,k_2,\ldots,k_m \le k^*, \\ \sum_{j=1}^m \tilde\mu_{k_j}+\alpha_r\mu_1=\tilde\mu_n}}\mathcal F_{r,m}(q_{k_1},q_{k_2},\ldots,q_{k_{m}}). 
\eeq

Clearly, the right-hand side of \eqref{finitesum} is a part of the sum in $\mathcal J_n$, and the converse is also true thanks to \eqref{d1}, \eqref{d2} and \eqref{d3} above. Thus, the sums on both sides of \eqref{finitesum} are the same.

\item In case of \eqref{Gef} and $n\ge 2$,  $\mathcal J_n$ is given by \eqref{Jnf}, and  relation \eqref{finitesum} under condition \eqref{sumcond} can be recast as
\beq\label{fsf}
\mathcal J_n=\sum_{r=1}^{N_*} \sum_{m=0}^{m^*} \sum_{\substack{2\le k_1,k_2,\ldots,k_m \le k^*, \\ \sum_{j=1}^m \tilde\mu_{k_j}+\alpha_r\mu_1=\tilde\mu_n}}\mathcal F_{r,m}(q_{k_1},q_{k_2},\ldots,q_{k_{m}}),
\eeq
for any non-negative integers $m^*,k^*$ satisfying  
\beq\label{scf} 
m^*\ge \tilde\mu_n/\tilde\mu_2\text{ and } k^*\ge n-1.
\eeq 
\end{enumerate}

\medskip
We are ready to prove Theorem \ref{thm3}  now.

\begin{proof}[Proof of Theorem \ref{thm3}]
We will prove for the case \eqref{Gex} first, and then make necessary changes for the case \eqref{Gef} later.

\medskip
\noindent\underline{Part A: Proof for the case of \eqref{Gex}.}
For any $N\in \N$, we denote by $(\mathcal T_N)$ the following statement: {\it There exist $\R^d$-valued polynomials $q_1(t)$, $q_2(t)$, \dots, $q_N(t)$ such that equation \eqref{qneq} holds true for $n=1,2,\ldots, N$, and  
 \beq\label{inhypo}
\Big|y(t) - \sum_{n=1}^N q_n(t)e^{-\mu_n t}\Big|=\bigo(e^{-(\mu_N+\delta_N)t})\quad \text{as }t\to\infty,
\eeq
for some $\delta_N>0$. }

We will prove $(\mathcal T_N)$  for all $N\in\N$ by induction in $N$.  

\medskip
\textbf{First step}  (N=1). By Theorem \ref{thmD2} and the fact $\mu_1=\lambda_*$, the statement $(\mathcal T_1)$ is true with $q_1(t)=\xi_*$ for all $t\in\R$, and some $\delta_1>0$.

\medskip\textbf{Induction step.} Let $N\ge 1$. Suppose there are polynomials $q_n$'s for $1\le n\le N$ such that the statement $(\mathcal T_N)$ holds true. 

For $n=1,\ldots, N$, let $y_n(t)=q_n(t) e^{-\mu_n t}$, $u_n(t)=y(t)-\sum_{k=1}^n y_k(t)$.
By induction hypotheses,  the polynomials $ q_n$'s satisfy \eqref{q1Eq}, \eqref{qneq} and
\beq\label{uNbo}
  u_N(t)= \bigo(e^{-(\mu_N + \delta_N)t}).
\eeq

Let $w_N(t)=e^{\mu_{N+1} t}u_N(t)$. We derive the differential equation for $w_N(t)$. 
\begin{align*}
w_N'-\mu_{N+1}w_N&=u_N'e^{\mu_{N+1} t}=(y'-\sum_{k=1}^N y_k')e^{\mu_{N+1} t}
=(-Ay + F(y) -\sum_{k=1}^N y_k' )e^{\mu_{N+1} t}\\ 
&=\Big(-Au_N -\sum_{k=1}^N A y_k + F(y) -\sum_{k=1}^N y_k' \Big)e^{\mu_{N+1} t} .
\end{align*}
Thus 
\beq\label{eqN}
w_N'+(A-\mu_{N+1}I_d) w_N=e^{\mu_{N+1} t} F(y) -e^{\mu_{N+1} t} \sum_{k=1}^N (A y_k+y_k').
\eeq

By \eqref{mulim}, we can choose a number $r_*\in\N$ such that 
\beq \label{pchoice}
\beta_{r_*}\ge \mu_{N+1}/\mu_1, \text{ which is equivalent to }
\alpha_{r_*}\ge \tilde\mu_{N+1}/\mu_1.
\eeq 

By \eqref{Ner},  one has 
\beq\label{Fcut}
F(x)=\sum_{r=1}^{r_*} F_r (x)+\bigo(|x|^{\beta_{r_*}+\varep_{r_*} }) \text{ as }x\to 0.
\eeq

Using \eqref{Fcut} with $x=y(t)$ and utilizing property \eqref{vraw}, we write the first term on the right-hand side of \eqref{eqN} as 
\beqs
e^{\mu_{N+1} t}F(y(t))
=e^{\mu_{N+1} t}\sum_{r=1}^{r_*} F_r (y(t))+e^{\mu_{N+1} t}\bigo(|y(t)|^{\beta_{r_*}+\varep_{r_*} })
=E(t)+e^{\mu_{N+1} t}\bigo(e^{-\lambda_*(\beta_{r_*}+\varep_{r_*})t}),
\eeqs
where 
\beq\label{T1T2}
E(t)=e^{\mu_{N+1} t}\sum_{r=1}^{r_*} F_r (y(t)).
\eeq
Because of condition for $\beta_{r_*}$ in \eqref{pchoice}, we then have 
\beq\label{Fsplit}
e^{\mu_{N+1} t}F(y(t))=E(t)+\bigo(e^{-\tilde\delta_N t}),\text{ where $\tilde\delta_N=\lambda_*\varep_{r_*}$.}
\eeq

The term $\sum_{r=1}^{r_*} F_r(y)$ in \eqref{T1T2} will be calculated as below.
For $k=1,\ldots,N$, denote
\beqs
\tilde y_k(t)= y_k(t) e^{\lambda_* t} = q_k(t) e^{-\tilde \mu_k t} \text { and } \tilde u_k(t)= u_k(t) e^{\lambda_* t}.
\eeqs

When $2\le k\le N$, one has  
\beq\label{Ovktilde1}
\tilde y_k(t)= q_k(t) e^{-\tilde \mu_k t} =\bigo(e^{-(\tilde\mu_k-\varepsilon)t}) \text{ for any $\varepsilon \in(0,\tilde\mu_k)$.}
\eeq
  
By \eqref{uNbo},
 \beq\label{Ountilde1}
\tilde u_N(t)= u_N(t) e^{\lambda_* t}= \bigo(e^{-(\tilde\mu_N + \delta_N)t} ).
\eeq
Also, from $(\mathcal T_{1})$, we similarly have 
\beq\label{util1}
\tilde u_1(t)= u_1(t) e^{\lambda_* t}=\bigo(e^{-\delta_1 t} ).
\eeq

Then
\beq\label{Frv}
F_r(y(t))= F_r(y_1 + u_1)
= F_r\big(e^{-\lambda_* t}(\xi_* + \tilde u_1 ) \big)
= e^{-\beta_r\lambda_* t} F_r(\xi_* + \tilde u_1).
\eeq

Let $s_*\in\N$ satisfy 
\beq \label{schoice} 
s_*\delta_1 + \beta_1\lambda_*\geq \mu_{N+1}
\text{ and } s_*\ge \tilde\mu_{N+1}/\tilde\mu_2. 
\eeq 

By Taylor's expansion \eqref{Taylor} with $s=s_*$, using the notation in \eqref{Frm},   
\beq\label{Frxi}
F_r (\xi_* + \tilde u_1)
 = \sum_{m=0}^{s_*} \mathcal F_{r,m}\tilde u_1^{(m)}+ g_{r,s_*}(\tilde u_1).
\eeq

It follows \eqref{Frv} and \eqref{Frxi} that
\beq\label{Fvex1}
F_r(y(t))=e^{-\beta_r\lambda_* t}\left(F_r(\xi_*)+\sum_{m=1}^{s_*} \mathcal F_{r,m}\tilde u_1^{(m)}\right)
+e^{-\beta_r\lambda_* t} g_{r,s_*}(\tilde u_1).
\eeq 

The terms in \eqref{Fvex1} are further calculated as follows. 

\medskip
For the last term in \eqref{Fvex1}, by using \eqref{grs}, \eqref{util1} and the first condition in \eqref{schoice}, we find that
\begin{align}\label{eg}
e^{-\beta_r\lambda_* t}  g_{r,s_*}(\tilde u_1)
&=e^{-\beta_r\lambda_* t}  \bigo(|\tilde u_1(t)|^{s_*+1})
=e^{-\beta_r\lambda_* t}  \bigo(e^{-\delta_1(s_*+1)t}) \notag\\
&=  \bigo(e^{-(\beta_1\lambda_*+\delta_1s_*+\delta_1)t})=\bigo ( e^{-(\mu_{N+1}+\delta_1)t}).
\end{align}

\medskip
For the remaining terms on the right-hand side of \eqref{Fvex1}, we write 
\begin{align}
\mathcal F_{r,m}\tilde u_1^{(m)}
&=\mathcal F_{r,m}\Big(\sum_{k=2}^N \tilde y_k +\tilde u_N\Big)^{(m)}
=\mathcal F_{r,m}\Big(\sum_{k=2}^N \tilde y_k +\tilde u_N,\sum_{k=2}^N \tilde y_k +\tilde u_N,\ldots,\sum_{k=2}^N \tilde y_k +\tilde u_N\Big) \notag\\
&=\mathcal F_{r,m}\Big(\sum_{k=2}^N \tilde y_k\Big)^{(m)}+\sum_{\rm finitely\ many}\mathcal F_{r,m}(z_1,\ldots,z_N). \label{u1m}
\end{align}

Note, in the case $N=1$, that the sum $\sum_{k=2}^N \tilde y_k$ and, hence, the term  $\mathcal F_{r,m}(\sum_{k=2}^N \tilde y_k)^{(m)} $ are not present in the calculations in \eqref{u1m}.
In the last sum of \eqref{u1m}, each $z_1,\ldots,z_N$ is either $\sum_{k=2}^N \tilde y_k$ or $\tilde u_N$, and at least one of $z_j$'s must be $\tilde u_N$.
By inequality \eqref{multineq}, estimate \eqref{Ovktilde1} for $\tilde y_k$, and estimates \eqref{Ountilde1}, \eqref{util1} for $\tilde u_N$, we have
\beqs
|\mathcal F_{r,m}(z_1,\ldots,z_N)|\le \|\mathcal F_{r,m}\| \cdot |z_1| \ldots |z_N| = \bigo(|\tilde u_N|)= \bigo(e^{-(\tilde\mu_N+\delta_N)t} ). 
\eeqs
Therefore,
\begin{align*}
\sum_{m=0}^{s_*} \mathcal F_{r,m}\tilde u_1^{(m)}
&=F_r(\xi_*) + \sum_{m=1}^{s_*} \mathcal F_{r,m}\Big(\sum_{k=2}^N \tilde y_k\Big)^{(m)} + \bigo(e^{-(\tilde\mu_N+\delta_N)t} )\\
&=\sum_{m=0}^{s_*} \sum_{k_1,\ldots,k_m \ge 2}^N\mathcal F_{r,m}(\tilde y_{k_1},\tilde y_{k_2},\ldots, \tilde y_{k_m})
+ \bigo(e^{-(\tilde\mu_N+\delta_N)t} )\\
&=\sum_{m=0}^{s_*} \sum_{k_1,\ldots,k_m \ge 2}^N e^{-t\sum_{j=1}^m \tilde\mu_{k_j}}\mathcal F_{r,m}(q_{k_1},q_{k_2},\ldots, q_{k_m}) 
+  \bigo(e^{-(\tilde\mu_N+\delta_N)t} ). 
\end{align*}

Thus,
\beq\label{esumF}
\begin{aligned}
e^{-\beta_r\lambda_* t} \sum_{m=0}^{s_*} \mathcal F_{r,m}\tilde u_1^{(m)}
&=\sum_{m=0}^{s_*} \sum_{k_1,\ldots,k_m=2}^Ne^{-t(\sum_{j=1}^m \tilde\mu_{k_j}+\beta_r\lambda_*)}\mathcal F_{r,m}(q_{k_1},q_{k_2},\ldots, q_{k_m})\\
&\quad +\bigo(e^{-(\tilde\mu_N+\beta_r\lambda_*+\delta_N)t} ). 
\end{aligned}
\eeq 

Again, in the case $N=1$, the last double summation has only one term corresponding to $m=0$, which is $F_r(\xi_*)$.

Using property \eqref{muN}, we have 
\beqs 
\tilde\mu_N + \beta_r\lambda_*+ \delta_N=
\tilde\mu_N + \alpha_r\lambda_*+\lambda_*+ \delta_N\ge \tilde \mu_{N+1}+\lambda_*+\delta_N=\mu_{N+1}+\delta_N. 
\eeqs 
Hence, the last term in \eqref{esumF} can be estimated as 
\beq\label{ee}
\bigo(e^{-(\tilde\mu_N + \beta_r\lambda_* +\delta_N )t} )= \bigo(e^{-(\mu_{N+1} +\delta_N )t} ).
\eeq 

Therefore, by formula of $E(t)$ in \eqref{T1T2}, and \eqref{Fvex1}, \eqref{eg}, \eqref{esumF}, \eqref{ee}, we have 
\beq\label{Tmore}
E(t)= e^{\mu_{N+1} t} \Big(J+\bigo(e^{-(\mu_{N+1} +\delta_N )t} )+\bigo(e^{-(\mu_{N+1} +\delta_1 )t} )\Big)
=e^{\mu_{N+1} t}J+\bigo(e^{-\min\{\delta_1,\delta_N\} t} ),
\eeq
where
\beq\label{Jdef}
J=\sum_{r=1}^{r_*}\sum_{m=0}^{s_*} \sum_{k_1,\ldots,k_m \ge 2}^Ne^{-t(\sum_{j=1}^m \tilde\mu_{k_j}+\beta_r\lambda_*)}\mathcal F_{r,m}(q_{k_1},q_{k_2},\ldots, q_{k_m}). 
\eeq

Denote $\mu= \tilde\mu_{k_1}+\ldots +\tilde\mu_{k_m}+\alpha_r\lambda_* $. 
When $m=0$, one has $\mu=\alpha_r\lambda_*$, which belongs to  $\tilde{S}$.
When $m\ge 1$, by property \eqref{Sprop}, $\mu$ also belongs to $\tilde{S}$.
Clearly, $\mu>0=\tilde \mu_1$.
Thus, in both cases of $m$,  the number $\mu$ must equal $\tilde\mu_p$ for a unique $p\ge 2$. 
Because of the indices $r,m,k_1,\ldots,k_m$ being finitely many, there are only finitely many such numbers $p$'s. Thus, there is  $p_*\in\N$ such that any index $p$ above satisfies $p\le p_*$.
Hence, the exponent in \eqref{Jdef} is
\beq \label{mul}
\sum_{j=1}^m \tilde\mu_{k_j}+\beta_r\lambda_*=\mu+\lambda_*=\tilde\mu_p+\lambda_*=\mu_p  \quad \text{ for some integer }  p\in[2,p_*].
\eeq

Using index $p$ in \eqref{mul}, we can split the sum in $J$ into two parts corresponding to $p \le N+1$ and $p\ge N+2$.
We then write  $J=S_1+S_2$,
where 
\begin{align*}
S_1&=\sum_{p=2}^{N+1}\sum_{r=1}^{r_*}\sum_{m=0}^{s_*}  \sum_{\substack{2\le k_1,\ldots,k_m\le N,\\  \sum_{j=1}^m \tilde\mu_{k_j}+\beta_r\lambda_*=\mu_p}}
e^{-\mu_p t} \mathcal F_{r,m}(q_{k_1},q_{k_2},\ldots, q_{k_m}),\\
S_2&=\sum_{p=N+2}^{p_*}\sum_{r=1}^{r_*}\sum_{m=0}^{s_*} \sum_{\substack{2\le k_1,\ldots,k_m\le N, \\ \sum_{j=1}^m \tilde\mu_{k_j}+\beta_r\lambda_*=\mu_p}} e^{-\mu_p t}\mathcal F_{r,m}(q_{k_1},q_{k_2},\ldots, q_{k_m}).
\end{align*}

We re-write $S_1=\sum_{k=2}^{N+1} e^{-\mu_k t} J_k$, where 
\beq\label{Jk}
J_k=\sum_{r=1}^{r_*} \sum_{m=0}^{s_*} \sum_{\substack{2\le k_1,\ldots,k_{m} \le N,\\ \sum_{j=1}^m \tilde\mu_{k_j}+\beta_r\lambda_*= \mu_k}} \mathcal F_{r,m}(q_{k_1},q_{k_2},\ldots, q_{k_{m}}) \text{ for $k=1,2,\dots,N+1$.}
\eeq

We estimate $S_2$. Set $\delta_N'=\min\{\tilde\delta_N,\delta_1,\delta_N,(\mu_{N+2}-\mu_{N+1})/2\}>0$.
Using inequality \eqref{multineq} to estimate $|\mathcal F_{r,m}(q_{k_1},q_{k_2},\ldots, q_{k_m})|$, and recalling that $q_{k_j}$'s are polynomials in $t$, we have 
\beqs
|\mathcal F_{r,m}(q_{k_1},q_{k_2},\ldots, q_{k_m})|
\le \|\mathcal F_{r,m}\| \cdot|q_{k_1}|\cdot |q_{k_2}|\ldots |q_{k_m}|  =\bigo(e^{\delta_N' t}).
\eeqs
For $e^{-\mu_p t}$, we use $\mu_p\ge \mu_{N+2}$, and obtain 
\beq\label{S2f}
S_2=\bigo( e^{-\mu_{N+2}t}e^{\delta_N' t})=\bigo( e^{-(\mu_{N+1}+\delta_N')t}).
\eeq 

Combining the above calculations from \eqref{Tmore} to \eqref{S2f} gives
\beq\label{T1e}
E(t) =e^{\mu_{N+1} t}\sum_{k=2}^{N+1}e^{-\mu_k t}J_k+\bigo( e^{-\delta_N' t}).
\eeq

Thus, by \eqref{eqN}, \eqref{Fsplit} and \eqref{T1e},
\beqs
w_N'+(A-\mu_{N+1}I_d) w_N
= \Big(  \sum_{k=2}^{N+1} e^{-\mu_k t}J_k-\sum_{k=1}^N (A y_k+y_k') \Big)e^{\mu_{N+1} t}
+\bigo( e^{-\delta_N' t}).
\eeqs

Using the fact $Ay_k+y_k'=e^{-\mu_k t} (q_k'+(A-\mu_k I_d)q_k)$, for $k=1,2,\ldots,N$, we deduce
\beq \label{wndiff}
w_N'+(A-\mu_{N+1}I_d) w_N
=-e^{\mu_{N+1} t}\sum_{k=1}^N e^{-\mu_k t }\chi_k + J_{N+1}
+\bigo( e^{-\delta_N' t}),
\eeq 
where
\beqs
\chi_1=q_1'+(A-\mu_1 I_d)q_1,\quad \chi_k= q_k'+(A-\mu_k I_d)q_k-J_k\text{ for } 2\le k\le N.
\eeqs

We already know $\chi_1=0$. Let us focus on the sum $\sum_{k=1}^N e^{-\mu_k t}\chi_k$ on the right-hand side of \eqref{wndiff}.
In case $N=1$, this sum is already zero.

Consider $N\ge 2$.  Note that condition $\sum_{j=1}^m \tilde\mu_{k_j}+\beta_r\lambda_*= \mu_k$ in formula \eqref{Jk} of $J_k$ is equivalent to 
$\sum_{j=1}^m \tilde\mu_{k_j}+\alpha_r\lambda_*= \tilde\mu_k$.
Then, for each $k=1,2,\ldots,N+1$, by the virtue of relation \eqref{finitesum} for $n=k\le N+1$, $r^*=r_*$, $m^*=s_*$ and $k^*=N$, one has
\beq\label{Jkrel}
J_k= \mathcal J_k \text{ for $k=1,2,\ldots,N+1$ .}
\eeq
Above, condition \eqref{sumcond} is met thanks to the condition for $\alpha_{r_*}$ in \eqref{pchoice}, the second condition for $s_*$ in \eqref{schoice}, and the fact $N\ge k-1$.

Thanks to \eqref{Jkrel} and  the induction hypothesis, $\chi_k=0$ for $2\le k\le N$.
Hence, \eqref{wndiff} becomes 
\beq\label{diffwn}
\begin{aligned}
&w_N'+(A-\mu_{N+1}I_d) w_N=J_{N+1}+\bigo( e^{-\delta_N' t}).
\end{aligned}
\eeq

Note that  $\mu_{N+1}> \mu_1 \ge \lambda_1$.
Let $\lambda_i$ is an eigenvalue of $A$ with $\lambda_i<\mu_{N+1}$. 
If $\lambda_i \le \lambda_{n_0}=\mu_1$ then $\lambda_i\le \mu_N$.
If $\lambda_i>\lambda_{n_0}$, then, according to property \eqref{ldS}, $\lambda_i\in S$,  hence, by the constraint  $\lambda_i<\mu_{N+1}$, we have $\lambda_i\le \mu_N$.
Therefore, in both cases  
\beqs
 e^{(\lambda_i-\mu_{N+1})t }|w_N(t)| = e^{\lambda_it }|u_N(t)| = e^{\lambda_i t}\bigo(e^{-(\mu_N+\delta_N)t}) =\bigo(e^{-\delta_N t}).
\eeqs
That is, condition \eqref{ellams} is satisfied.

Applying Lemma \ref{exp-odelem} to the equation \eqref{diffwn}, there exists polynomial $q_{N+1}:\R\to\R^d$ and a number $\delta_{N+1}>0$ such that
\beq\label{WN}
|w_N(t)-q_{N+1}(t)|=\bigo(e^{-\delta_{N+1} t}).
\eeq
Moreover $q_{N+1}(t)$ solves
\beqs
q_{N+1}'+(A-\mu_{N+1}I_d)q_{N+1} = J_{N+1}=\mathcal J_{N+1}, 
\eeqs
that is, equation \eqref{qneq} holds for $n=N+1$.

Multiplying \eqref{WN} by $e^{-\mu_{N+1}t}$ gives
 \beqs
\Big|y(t) - \sum_{n=1}^{N+1} q_n(t)e^{-\mu_n t}\Big|=\bigo(e^{-(\mu_{N+1}+\delta_{N+1})t}),
\eeqs
which proves \eqref{inhypo} for $N:=N+1$.

Hence the statement  $(\mathcal T_{N+1})$ holds true.

\medskip\textbf{Conclusion for Part A.} By the induction principle, the statement $(\mathcal T_{N})$ holds true for all $N\in\N$.
Note also that, the polynomials $(\mathcal T_{N+1})$ are exactly the ones from $(\mathcal T_{N})$. Hence, the polynomials $q_n$'s exist for all $n\in\N$, for which $(\mathcal T_{N})$ holds true for all $N\in\N$.
Therefore, we obtain the desired expansion \eqref{vx}.

\medskip
\noindent\underline{Part B: Proof for the case of \eqref{Gef}.} 
We follow the  proof in Part A with the following adjustments.
The number $r_*$ is simply $N_*$, and condition \eqref{pchoice} for $r_*$  is not required anymore. All the sum $\sum_{r\ge 1}$ appearing in the proof that involves $F_r$ or $\mathcal F_{r,m}$  will be replaced with $\sum_{1\le r\le N_*}$. From \eqref{Fcut} to the end of the proof in Part A, positive number $\varep_{r_*}$ is arbitrary, and number $\beta_{r_*}$ in calculations from \eqref{Fcut} to \eqref{Fsplit} is replaced with any number $\beta_*\ge \mu_{N+1}/\mu_1$. Then \eqref{Fcut} still holds true thanks to \eqref{errF2}.
We also take into account that $\mathcal J_n$ is given by \eqref{Jnf}, and one has relation \eqref{fsf} under condition \eqref{scf}.
With these changes, the above proof in Part A goes through, and we obtain the desired statement for this case \eqref{Gef}.

\medskip    
The proof of Theorem \ref{thm3} is now complete.
\end{proof}

\begin{remark}\label{rmk44}
 Assume we have \eqref{Gef}, then by adding more functions $F_k=0$ and numbers $\beta_k$'s, for $k>N_*$, such that $\beta_k$ increases strictly to infinity, one can convert \eqref{Gef} into \eqref{Gex}. (For example, one can take $\beta_k=\beta_{N_*}+k$ for $k>N_*$.) However, we did not use this fact in Part B of the proof of Theorem \ref{thm3} above. The reason is to have simpler constructions of $\widetilde S$ and $q_n$'s in \eqref{tilfS} and \eqref{Jnf} for the case \eqref{Gef}, as opposed to  \eqref{deftilS} and \eqref{qneq} if it is converted to \eqref{Gex}. 
\end{remark}

\section{Extended results}\label{finapprox}

In this section, we extend Theorem \ref{thm3} to the situations that require less of the function $F$. 

First, we consider the case when the function $F$ in \eqref{vtoy} only has a finite sum approximation. We will find a finite sum  asymptotic approximation for decaying solutions of \eqref{vtoy}.

Assume function $F$ satisfies \ref{FL} and \ref{h2} of Assumption \ref{assumpG}  with \eqref{errF2} being replaced with 
\beq\label{errF3}
\Big|F(x)-\sum_{k=1}^{N_*} F_k(x)\Big|=\bigo(|x|^{\beta_{N_*}+\bar\varep}) \text{ as $x\to 0$, for some number  $\bar\varep>0$.} 
\eeq

Note that \eqref{errF3} is different from \eqref{errF2} due to the restriction of $\bar\varep$.
Also, we usually think of $\bar\varep$ as a small number, but, in \eqref{errF3}, it can be large. This happens when the remainder $F(x)-\sum_{k=1}^{N_*} F_k(x)$ may have very precise approximation, i.e., large $\bar\varep$, but it does not have a homogeneous structure that we can take advantage of. 

From \eqref{errF3}, one can see that estimate \eqref{Ner} still holds for all $N\in\N\cap [1,N_*]$, where $\delta_N$ in any number in $(0,\beta_{N+1}-\beta_N)$ when $N<N_*$, and is $\bar\varep$ when $N=N_*$. Consequently, \eqref{Gyy} is still valid, and the facts and results in section \ref{secD} apply.

Let $y(t)$ be a non-trivial, decaying solution of \eqref{vtoy}.
Applying Theorem \ref{thmD2}, we have the first approximation \eqref{firstrate}.

For more precise approximations, define sets $\tilde S$ and $S$ by \eqref{tilfS} and \eqref{defMu1}, respectively.

Let $\bar N\in\N$ be defined by 
\beq\label{opNbar}
\bar N=\max\{N\in\N: \lambda_*(\beta_{N_*}+\bar\varep)>\mu_N\}.
\eeq

From the definition of $ \tilde S$, we see that $\alpha_{N_*} \lambda_*  \in  \tilde S$. Therefore, there exists a unique number $N'\in \N$ such that  $\alpha_{N_*} \lambda_*  = \tilde \mu_{N'}$, which is equivalent to 
$\mu_{N'} =  \beta_{N_*} \lambda_*$. The last expression gives $\mu_{N'}>\lambda_*=\mu_1$,  thus, one must have $N'\ge 2 $.
Note that $N'$ belongs to the set on the right-hand side of \eqref{opNbar}, then $\bar N\ge N'\ge 2$.

We obtain the finite approximation for decaying solutions under the assumption  \eqref{errF3} as follows.
\begin{theorem}\label{thm51}
There exist $\R^d$-valued polynomials $q_n(t)$'s, for $1\le n\le \bar N$, and a number $\delta>0$ such that  
\beq \label{finapp}
\Big|y(t)-\sum_{n=1}^{\bar N} q_n(t) e^{-\mu_n t}\Big|=\bigo(e^{-(\mu_{\bar N}+\delta)t}),
\eeq
where each polynomial $q_n(t)$, for $1\le n\le \bar N$, satisfies equation
 \beq\label{qneqf}
q_{n}'+(A-\mu_{n}I_d)q_{n} 
=\sum_{r=1}^{N_*} \sum_{\substack{m \ge 0,k_1,k_2,\ldots,k_{m} \ge 2,\\  
\sum_{j=1}^{m} \tilde\mu_{k_j}+\alpha_r\lambda_*=\tilde\mu_n}} \mathcal F_{r,m}(q_{k_1},q_{k_2},\ldots,q_{k_{m}})
\text{ in $\R$.}
\eeq
\end{theorem}

\begin{proof}
We follow Part A of the proof of Theorem \ref{thm3}, with some changes similar to those in Part B.

First, we take $r_*=N_*$, $1\le r\le N_*$ and replace $\varep_{r_*}$ with number $\bar\varep$ in \eqref{errF3}.  

Second, we replace condition \eqref{pchoice} with
$\lambda_*(\beta_{r_*}+\bar\varep)>\mu_{\bar N}$,
which is satisfied by definition of $\bar N$ in \eqref{opNbar}. 

Third, for $1\le N\le \bar N-1$, the calculations \eqref{Fcut}--\eqref{Fsplit} are still valid with number $\tilde\delta_N$ in \eqref{Fsplit} being changed to 
$\tilde\delta_N=\lambda_*(\beta_{r_*}+\bar\varep)-\mu_{N+1}$.
Note that $\tilde\delta_N\ge \lambda_*(\beta_{r_*}+\bar\varep)-\mu_{\bar N}>0$.

We do finite induction in $N$ for $1\le N\le \bar N$ and obtain $(\mathcal T_{\bar N})$, which, by \eqref{inhypo}, yields
\eqref{finapp}. Here, each polynomial $q_n(t)$, for $1\le n\le \bar N$, satisfies equation \eqref{qneq} with $\mathcal J_n$ being given by \eqref{Jnf} particularly; that is, we obtain equation \eqref{qneqf}.
\end{proof}

Next, we relax the regularity requirements for $F$ and $F_k$'s.

Regarding $F$, its local Lipschitz property is imposed to guarantee the existence and uniqueness of solutions at least starting with small initial data. However, in some problems, $F$ is not that regular, but a small solution $y(t)$, for $t\in[0,\infty)$, already exists and is given. Then our results obtained above apply to this solution $y(t)$.

Regarding $F_k$'s, what we need in the proofs of Theorems \ref{thm3} and \ref{thm51} is that each $F_k$, in addition to being positively homogeneous,  has the Taylor series approximation of all orders about $\xi_*$, where $\xi_*$ is from Theorem \ref{thmD2}.
Because  $\xi_*$ depends on $y(t)$ and varies in $\R^d_0$, function $F_k$ is required in Assumption \ref{assumpG} to be smooth on the entire set $\R^d_0$. However, in many cases, $F_k$ is only known to be smooth on an open set $V$ strictly smaller than $\R^d_0$. Then one needs $\xi_*$ to belong to $V$ as well. This is possible when more information about $\xi_*$, as an eigenvector of matrix $A$, is provided.

These two points will be reflected in Theorem \ref{thm57} below.

\begin{definition}\label{notation2}
For an open set $V$ in $\R^d$, denote by $\mathcal X(V)$, respectively $\mathcal X^0(V)$, the set of locally Lipschitz  continuous, respectively  continuous, functions on $\R^d$, with approximation \eqref{Gex} or \eqref{Gef}  where $F_k\in \mathcal H_{\beta_k}(\R^d)\cap C^\infty(V)$
for all respective $k$'s.
The sets $\widehat{\mathcal X}(V)$ and $\widehat{\mathcal X}^0(V)$ are defined similarly with \eqref{errF3} replacing \eqref{Gex}   and \eqref{Gef}.
In particular, denote $\mathcal X=\mathcal X(\R^d_0)$ and $\mathcal X^0=\mathcal X^0(\R^d_0)$.
\end{definition}

Note that $\mathcal X$ is the set of functions that satisfy Assumption \ref{assumpG}. 

An extension of the results in Theorems \ref{thm3} and \ref{thm51} is the following general theorem.

\begin{theorem}\label{thm57}
Suppose that all eigenvectors of matrix $A$ belong to an open set $V$ in $\R^d$. 
\begin{enumerate}[label=\tnum]
 \item Then Theorem \ref{thm3} applies to any function $F\in\mathcal X(V)$, and Theorem \ref{thm51} applies to any function $F\in \widehat{\mathcal X}(V)$, for any non-trivial, decaying solution $y(t)$ of \eqref{vtoy}. 
 
 \item  If $F\in \mathcal X^0(V)$, respectively $F\in \widehat{\mathcal X}^0(V)$, then Theorem \ref{thm3}, respectively Theorem \ref{thm51}, still holds true for a solution $y(t)\in C^1([0,\infty))$ of \eqref{vtoy} that satisfies
 $y(t)\to 0$ as $t\to\infty$, and there is a divergent, strictly increasing sequence $(t_n)_{n=1}^\infty$ in $(0,\infty)$ such that   $y(t_n)\ne 0$ for all $n\in\N$.
\end{enumerate}
\end{theorem}
\begin{proof}
(i) In the proofs of Theorems \ref{thm3} and \ref{thm51}, the eigenvector $\xi_*$ belongs to $V$, and, thanks to the condition $F_k\in C^\infty(V)$, we can still use the Taylor expansions of $F_k$'s about $\xi_*$. Therefore, both proofs are unchanged and produce respective conclusions.

(ii) We re-examine Proposition \ref{nzprop}. 
Select $T_\varep=t_n$ for sufficiently large $n$ such that \eqref{Te} still holds. Then we still obtain upper bound \eqref{yexp1}.
With $y(T_\varep)=y(t_n)\ne 0$, the estimate \eqref{ylowexp} holds for some $C'_2>0$. Thus, the inequality in \eqref{yexp2} holds for all $t\ge T_\varep$. With such a lower bound of $|y(t)|$, we can still prove Theorem \ref{thmD2}.
After that, the argument in (i) continues to be valid.
\end{proof}

The sets defined in Definition \ref{notation2} and used in Theorem \ref{thm57} will be explored more in section \ref{specific} below. Here, we state their very first property.

\begin{proposition}\label{prop52}
 For any open set $V$ in $\R^d$, the sets  $\mathcal X(V)$, $\widehat{\mathcal X}(V)$ $\mathcal X^0(V)$ and $\widehat{\mathcal X}^0(V)$ are  linear spaces. 
 \end{proposition}
\begin{proof}
We gives a proof for $\mathcal X(V)$, the other sets can be proved similarly.  Thanks to Remark \ref{rmk44}, it suffices to prove that the sum of any two functions of the form \eqref{Gex} is also of the form \eqref{Gex}.
 Suppose $F(x)$ is the same as in \eqref{Gex}, and 
 \beq\label{aG}
 G(x)\sim \sum_{k=1}^\infty G_k(x),\eeq 
where each $G_k$ is similar to $F_k$, but with degree $\beta'_k>1$ instead of $\beta_k$.
Arrange the set $\{\beta_k,\beta'_j:k,j\in\N\}$ as an strictly increasing sequence $(\bar\beta_k)_{k=1}^\infty$.
Clearly, $\bar\beta_k\to\infty$ as $k\to\infty$, and  $(\beta_k)_{k=1}^\infty$ and $(\beta_k')_{k=1}^\infty$ are subsequences of $(\bar \beta_k)_{k=1}^\infty$.
By inserting the zero function into \eqref{Gex} and \eqref{aG} when needed, one can rewrite (and verify) $F$ and $G$ as 
 $$F(x)\sim \sum_{k=1}^\infty \tilde F_k(x)\text{ and } G(x)\sim \sum_{k=1}^\infty \tilde G_k(x),$$
 where $\tilde F_k(x)$ and $\tilde G_k(x)$ are in $C^\infty(V)$, positively homogeneous of the same degree $\bar\beta_k$.
 Then, $F+G$ is, obviously,  of the form \eqref{Gex} with $\widetilde F_k+\widetilde G_k$ replacing $F_k$, and   $\bar\beta_k$ replacing $\beta_k$.
\end{proof}

\section{Specific cases and examples}\label{specific}

We specify many cases for the function $F$ in Theorem \ref{thm57}, i.e., describe classes of functions in  the spaces $\mathcal X(V)$, $\widehat{\mathcal X}(V)$ $\mathcal X^0(V)$ and $\widehat{\mathcal X}^0(V)$ in Definition \ref{notation2}.

For $n\in\N$, $p\in[1,\infty)$ and $x=(x_1,x_2,\ldots,x_n)\in\R^n$,  the $\ell^p$-norm of $x$ is 
$$\|x\|_p=\Big (\sum_{j=1}^n|x_j|^p \Big)^{1/p}.$$

We recall that all these norms $\|\cdot\|_p$ on $\R^n$ are equivalent to each others.

For any $n\in\N$, $p\ge 1$ and $\alpha>0$, one has the following.
\begin{enumerate}[label=\rnum]
 \item\label{rall} The function $x\in\R^n\mapsto \|x\|_p^\alpha$ belongs to $C(\R^n)\cap C^\infty(\R_*^n)\cap \mathcal H_\alpha(\R^n)$.

 \item\label{reven} Assume, additionally, that $p$ is an even number. Then the function $x\in\R^d\mapsto \|x\|_p^\alpha$ belongs to $C^\infty(\R^n_0)$.
\end{enumerate}

The first class of functions in $\mathcal X$ we describe is in the next theorem, which involves the $\ell^p$-norms of $x$ and polynomials on $\R^d$.

\begin{theorem}\label{thm63}
Let $\delta>0$ and $m\in\N$. Suppose $G:(-\delta,\infty)^m\to \R$ be a $C^\infty$-function with $G(0)=0$, and
$G_0:\R^d\to\R^d$ is a homogeneous polynomial of degree $m_0\in \Z_+$.
Define a function $F:\R^d\mapsto \R^d$ by
\beq\label{Fmanyp}
F(x)=G(\|x\|_{p_1}^{s_1},\|x\|_{p_2}^{s_2},\ldots,\|x\|_{p_m}^{s_m})G_0(x) \text{ for $x\in\R^d$,}
\eeq
where $p_j\in[1,\infty)$ and $s_j\in(0,\infty)$ for $j=1,2,\ldots,m$, are given real numbers.

Let $\bar s=\min\{s_j:j=1,2,\ldots, m\}$. Assume $\bar s+m_0> 1$.
Then the following statements hold true.

\begin{enumerate}[label=\tnum]
 \item\label{p1} $F(0)=0$ and  $F\in C(\R^d)\cap C^\infty(\R_*^d)$. 

 \item\label{p2} $F\in \mathcal X^0(\R^d_*)$.
 
 \item\label{p3} If $p_1,\ldots,p_m>1$, then $F\in C^1(\R^d)$, and, consequently, $F$ is locally Lipschitz in $\R^d$. 
 
 \item\label{p4} If  $p_1,p_2,\ldots,p_m$ are even numbers, then  $F\in \mathcal X$.
\end{enumerate}

\end{theorem}
\begin{proof}
In part \ref{p1}, the property $F(0)=0$ follows the fact $G(0)=0$. The proof of the remaining statement in \ref{p1} is elementary, using the chain rule for derivatives and property \ref{rall} right before this theorem.

\medskip
We prove \ref{p2}. By using the Taylor expansion of $G(z)$, for $z\in (-\delta,\infty)^m$, about the origin of $\R^m$, we can approximate $G(\|x\|_{p_1}^{s_1},\|x\|_{p_2}^{s_2},\ldots,\|x\|_{p_m}^{s_m})$, for $k\in\N$, by 
\beqs
\sum_{\substack{\gamma=(\gamma_1,\gamma_2,\ldots,\gamma_m)\in \Z_+^m,\\|\gamma|\le k}} c_\gamma 
\|x\|_{p_1}^{s_1\gamma_1}\|x\|_{p_2}^{s_2\gamma_2}\ldots\|x\|_{p_m}^{s_m\gamma_m}
\eeqs
with the remainder being
\beqs
\bigo((\|x\|_{p_1}^{s_1}+\|x\|_{p_2}^{s_2}+\ldots+\|x\|_{p_m}^{s_m})^{k+1})=\bigo(|x|^{\bar s(k+1)})\text{ as $x\to 0$,}
\eeqs
where  each $\gamma$ is a multi-index with length
\beq \label{cg}
|\gamma|=\gamma_1+\gamma_2+\ldots+\gamma_m, \text{ and }
c_\gamma=\frac{1}{\gamma_1!\gamma_2!\ldots\gamma_m!}\cdot
\frac{\partial^{|\gamma|}G(0)}{\partial x_1^{\gamma_1} \partial x_2^{\gamma_2} \ldots\partial x_m^{\gamma_m}}.
\eeq 
Re-arrange the set 
$$\Big\{ m_0+\sum_{j=1}^m s_j \gamma_j:\gamma_j\in \Z_+,(\gamma_1,\gamma_2,\ldots,\gamma_m)\ne 0\Big\}$$ as a strictly increasing sequence $(\beta_k)_{k=1}^\infty$.
Note that $\beta_k\to\infty$ as $k\to\infty$, and, because of the assumption $\bar s+m_0>1$, we have $\beta_k>1$ for all $k\in\N$. 

Then we can re-write $F(x)$ in the form of \eqref{Gex}, where
\beq\label{FFk}
F_k(x)=\sum_{\substack{\gamma=(\gamma_1,\gamma_2,\ldots,\gamma_m)\in \Z_+^m,\\m_0+\sum_{j=1}^m s_j \gamma_j=\beta_k}} c_\gamma \|x\|_{p_1}^{s_1\gamma_1}\|x\|_{p_2}^{s_2\gamma_2}\ldots\|x\|_{p_m}^{s_m\gamma_m}G_0(x).
\eeq

By property \ref{rall} right before this theorem and property \ref{pm} after Definition \ref{phom}, $F_k\in \mathcal H_{\beta_k}(\R^d)\cap C^\infty(\R_*^d)$. By this and the facts $F(0)=0$, $F\in C(\R^d)$ in \ref{p1}, we conclude $F\in \mathcal X^0(\R_*^d)$. 

\medskip
We prove \ref{p3}.  Because $G_0$ is a homogeneous polynomial of degree $m_0$, there is $C>0$ such that $G_0(x)$ and its derivative matrix $DG_0(x)$ can be estimated, for any $x\in\R^d$, by   
\beq\label{Gb}
|G_0(x)|\le C|x|^{m_0}\text{ and }  
|DG_0(x)|\begin{cases}\le C |x|^{m_0-1}&\text{ if $m_0\ge 1$,}\\
                      =0& \text{ if $m_0=0$.}
         \end{cases}
\eeq

By using the linear approximation of $G(z)$ for $z$ near $0$ in $\R^m$, we have 
\beqs
G(z)=\bigo(|z|)=\bigo(|z_1|+\ldots+|z_m|), \text{ as } z=(z_1,\ldots,z_m)\to 0.
\eeqs

Applying this property to $z=(\|x\|_{p_1}^{s_1},\|x\|_{p_2}^{s_2},\ldots,\|x\|_{p_m}^{s_m})$, we have
\beqs
G(\|x\|_{p_1}^{s_1},\|x\|_{p_2}^{s_2},\ldots,\|x\|_{p_m}^{s_m})=\bigo(\|x\|_{p_1}^{s_1}+\|x\|_{p_2}^{s_2}+\ldots+\|x\|_{p_m}^{s_m})=\bigo(|x|^{\bar s}) \text{ as $x\to 0$,}
\eeqs
and, together with the first inequality in \eqref{Gb},
\beq\label{FO}
F(x)=\bigo(|x|^{\bar s+m_0})\text{ as $x\to 0$. }
\eeq

Since $\bar s+m_0>1$ and $F(0)=0$, it follows \eqref{FO} that 
\beq\label{DFz}
D F(0)=0.
\eeq

For $1\le i\le m$ and $1\le j\le d$, one has the partial derivative, thanks to $p_i>1$, 
$$x=(x_1,\ldots,x_d)\in\R^d\mapsto \frac{\partial (|x_j|^{p_i})}{\partial x_j}=p_i |x_j|^{p_i-1}\sign(x_j),
$$
which is a continuous function on $\R^d$. 

For $x\in\R^d\setminus\{0\}$ and $j=1,2,\ldots,d$, we have
\beq\label{Fxj}\begin{aligned}
\frac{\partial F(x)}{\partial x_j} 
&=\sum_{i=1}^m \frac{\partial G(z)}{\partial z_i}\Big|_{z=(\|x\|_{p_1}^{s_1},\|x\|_{p_2}^{s_2},\ldots,\|x\|_{p_m}^{s_m})}
s_i\|x\|_{p_i}^{s_i-p_i} |x_j|^{p_i-1}\sign(x_j)G_0(x)\\
&\quad +G(\|x\|_{p_1}^{s_1},\|x\|_{p_2}^{s_2},\ldots,\|x\|_{p_m}^{s_m})\frac{\partial G_0(x)}{\partial x_j}. 
\end{aligned}
\eeq 

Clearly, $\partial F(x)/\partial x_j$ is continuous on $\R^d\setminus\{0\}$.
Consider its continuity at the origin.

For the first summation on the right-hand side of \eqref{Fxj}, 
\beq\label{dF1}
 \frac{\partial G(z)}{\partial z_i}\Big|_{z=(\|x\|_{p_1}^{s_1},\|x\|_{p_2}^{s_2},\ldots,\|x\|_{p_m}^{s_m})}
=\bigo(1)\text{ as $x\to 0$ },
\eeq 
and, thanks to the first estimate in \eqref{Gb}, 
\beqs
\|x\|_{p_i}^{s_i-p_i} |x_j|^{p_i-1}|\sign(x_j)G_0(x)|\le \bigo(|x|^{s_i-1}|x|^{m_0})=\bigo(|x|^{\bar s+m_0-1})\text{ as $x\to 0$.}
\eeqs

By the second estimate in \eqref{Gb}, the last term in \eqref{Fxj}, it is zero when $m=0$, and can be estimated, when $m_0\ge 1$, by
\beq \label{dF3}
\Big|G(\|x\|_{p_1}^{s_1},\|x\|_{p_2}^{s_2},\ldots,\|x\|_{p_m}^{s_m})\frac{\partial G_0(x)}{\partial x_j}\Big| 
\le \bigo(|x|^{\bar s})C|x|^{m_0-1}=\bigo(|x|^{\bar s+m_0-1}) \text{ as $x\to 0$.}
\eeq 

The above estimates from \eqref{dF1} to \eqref{dF3} for the right-hand side of \eqref{Fxj} yield
\beqs
\lim_{x\to 0}\frac{\partial F(x)}{\partial x_j}=0. 
\eeqs
Together with \eqref{DFz}, this limit implies that $\partial F(x)/\partial x_j$ is continuous at the origin for $j=1,2,\ldots,d$. 
Therefore, $F\in C^1(\R^d)$, and, consequently, $F$ is locally Lipschitz in $\R^d$.

\medskip
Finally, we prove \ref{p4}. In case all $p_j$'s are even numbers, then,  by property \ref{reven} right before Theorem \ref{thm63}, all $F_k$'s in \eqref{FFk} belong to $C^\infty(\R^d_0)$.
Combining this fact with (ii) and (iii) above, we have  $F\in\mathcal X$.
\end{proof}

\begin{example}\label{Eg3}
Let $\alpha$ be any number in $(0,\infty)$ that is not an even integer, and
\beq\label{egz} F(x)=|x|^\alpha x \text{ for $x\in\R^d$}.
\eeq 
Applying Theorem \ref{thm63}\ref{p4} to $m=1$, $G(z)=z$ for $z\in\R$, $G_0(x)=x$, $p_1=2$ and $s_1=\alpha$, we have $F\in \mathcal X$.
Even in this simple case, the asymptotic expansions obtained in Theorem \ref{thm3} is new. 
\end{example}

\begin{example}\label{Eg5}
Given a constant $d\times d$ matrix $M_0$, even numbers $p_1,p_2\ge 2$, and real numbers  $\alpha,\beta>0$, let
\beq\label{eg2}
F(x)=\frac{\|x\|_{p_1}^\alpha M_0 x}{1+\|x\|_{p_2}^\beta}\text{ for $x\in\R^d$.}
\eeq 

Applying Theorem \ref{thm63}\ref{p4} to functions  $G(z_1,z_2)=z_1/(1+z_2)$, $G_0(x)=M_0x$ and numbers $s_1=\alpha$, $s_2=\beta$, one has $F\in\mathcal X$.
The explicit form of \eqref{Gex} can be obtained quickly as follows.

For $x\in\R^d$ with  $\|x\|_{p_2}<1$, we expand $1/(1+\|x\|_{p_2}^\beta)$, using the geometric series, and can verify that
\beq\label{eg3}
F(x) \sim \sum_{k=1}^\infty (-1)^{k-1} \|x\|_{p_1}^{\alpha}\|x\|_{p_2}^{(k-1)\beta}M_0x,
\eeq
in the sense of \ref{h1} in Assumption \ref{assumpG}. This yields \eqref{Gex} with $\beta_k=1+\alpha+(k-1)\beta$.

When $\|\cdot\|_{p_1}=\|\cdot\|_{p_2}=|\cdot|$, function $F$ in \eqref{eg2} covers the particular case discussed in \eqref{eg1}, and expansion \eqref{eg3} simply reads as
\beqs
F(x) \sim \sum_{k=1}^\infty (-1)^{k-1} |x|^{\alpha+(k-1)\beta}M_0x.
\eeqs
\end{example}

\begin{example}\label{Eg4}
For $k\in\N$, let $M_k$ be a constant $d\times d$ matrix, and $p_k\ge 2$ be an even number, and $\alpha_k> 0$.

(a) Each function $x\in\R^d\mapsto \|x\|_{p_k}^{\alpha_k} M_kx$ can play the role of $F_k$ in \eqref{Gex} or \eqref{errF3}.
In this case, we write, respectively, 
\beq\label{eg0}
F(x)\sim \sum_{k=1}^\infty \|x\|_{p_k}^{\alpha_k}M_k x, \text{ or } 
\Big|F(x)-\sum_{k=1}^{N_*} \|x\|_{p_k}^{\alpha_k}M_k x\Big|=\bigo(|x|^{\alpha_{N_*}+1+\bar\varep}) 
\text{ as $x\to 0$.} 
\eeq

In particular, thanks to Theorem \ref{thm63}\ref{p4}, the function 
\beqs
F(x)=\sum_{k=1}^{N_*} \|x\|_{p_k}^{\alpha_k}M_k x, \text{ for $x\in\R^d$, belongs to $\mathcal X$.} 
\eeqs

(b) We can replace $M_kx$ in \eqref{eg0} with  an $\R^d$-valued homogeneous polynomial in $x$ of degree $m_k\in\Z_+$.  Of course, the set $\{\alpha_k + m_k:k\in\N\}$ is required to be in $(1,\infty)$ and can be re-arranged as a sequence that strictly increases to infinity.
\end{example}

In Examples \ref{Eg3}, \ref{Eg5} and \ref{Eg4} above,  we can also consider more complicated variations.
For example, in \eqref{egz}, \eqref{eg2} and \eqref{eg0}, we can replace $|x|$ or $\|x\|_{p_k}$ with $\|S_k x\|_{p_k}$, where $S_k$'s are invertible $d\times d$ matrices.

\medskip
Note that a positively homogeneous  function of the form \eqref{heg}, in general, does not belong to $C^\infty(\R^d_0)$. Hence, it cannot play a role of an $F_k$ in \eqref{Gex} or \eqref{errF3}.
However, in some cases, see \eqref{65a} and \eqref{65b} below, it can.

\begin{theorem}\label{thm65}
Consider function $F(x)$ given by \eqref{heg} with $X=Y=\R^d$, $s\ge 1$ and $(Y_j,\|\cdot\|_{Y_j})=(\R^{n_j},\|\cdot\|_{p_j})$ for $j=1,\ldots,s$.
Suppose, for $j=1,\ldots,s$,
\beq\label{65a} \text{ the number $p_j$ is even, and}
\eeq 
\beq\label{65b}
\text{ the only solution of equation $P_j(x)=0$ is $x=0$.} 
\eeq 
\begin{enumerate}[label=\tnum]
\item\label{s1} One has $F\in \mathcal H_\beta(\R^d)\cap C(\R^d)\cap C^\infty(\R^d_0)$, where number $\beta$ is defined in \eqref{simF}. 

\item\label{s2} If $\beta>1$, then $F\in \mathcal X^0$.

\item\label{s3} Let $\bar\nu=\min\{\nu_j:j=1,\ldots,s\}$ and assume $m_0+\bar\nu>1$. Then $F\in C^1(\R^d)$.
Consequently, $F\in \mathcal X$.
\end{enumerate}
\end{theorem}
\begin{proof}
 For part \ref{s1},  the fact $F\in \mathcal H_\beta(\R^d)$ is due to \eqref{simF}, while the other fact
 $F\in C(\R^d)\cap C^\infty(\R^d_0)$ is clear.
 Part \ref{s2} comes from part (i).
 
We prove part \ref{s3} now. Same as \eqref{Gb}, there is $C>0$ such that, for $j=0,1,\ldots,s$,  and any $x\in\R^d$,
\beq\label{Pb}
|P_j(x)|\le C|x|^{m_j}\text{ and }  
|DP_j(x)|\begin{cases}\le C |x|^{m_j-1}&\text{ if $m_j\ge 1$,}\\
                      =0& \text{ if $m_j=0$,}
         \end{cases}
\eeq
 
 Because $s\ge 1$ and $m_j\ge 1$ for $j\ge 1$, we have $\beta=m_0+\sum_{j=1}^s m_j\nu_j\ge m_0+\bar\nu>1$.

 Note that $F(0)=0$ and, by the first estimate in \eqref{Pb}, 
 \beqs 
 F(x)=\bigo(|x|^{m_0+\sum_{j=1}^m \nu_j m_j}) =\bigo(|x|^\beta)\text{ as $x\to 0$.}
 \eeqs 
 Then, thanks to  $\beta>1$, we have the derivative matrix $DF(0)=0$.
 
 For $j=1,2,\ldots,s$, write $P_j=(P_{j,1},P_{j,2},\ldots,P_{j,n_j})$.

Let $x=(x_1,\ldots,x_d)\in \R^d\setminus\{0\}$. Then, thanks to condition\eqref{65b}, $P_j(x)\ne 0$ for $j=1,2,\ldots,s$.
For $i=1,2,\ldots,d$, we have the partial derivative 
 \begin{multline}\label{FPderiv}
\frac{\partial F(x)}{\partial x_i} 
=\|P_1(x)\|_{p_1}^{\nu_1} \|P_2(x)\|_{p_2}^{\nu_2} \ldots \|P_s(x)\|_{p_s}^{\nu_s} \frac{\partial P_0(x)}{\partial x_i}\\
+\left\{ \sum_{j=1}^s \left(\prod_{\substack{1\le j'\le s,\\ j'\not=j} }\|P_{j'}(x)\|_{p_{j'}}^{\nu_{j'}}\right) 
\nu_j\|P_j(x)\|_{p_j}^{\nu_j-p_j} \left(\sum_{\ell=1}^{n_j} (P_{j,\ell}(x))^{p_j-1} \frac{\partial P_{j,\ell}(x)}{\partial x_i}\right)\right\} P_0(x) . 
\end{multline}

One can see that this partial derivative is continuous on $\R^d\setminus\{0\}$.
For the continuity of $\partial F(x)/\partial x_j$ at the origin, we estimate the right-hand side of \eqref{FPderiv}.
On the one hand, 
\beqs 
\|P_1(x)\|_{p_1}^{\nu_1} \|P_2(x)\|_{p_2}^{\nu_2} \ldots \|P_s(x)\|_{p_s}^{\nu_s} \Big|\frac{\partial P_0(x)}{\partial x_j}\Big|
\text{ is zero if $m_0=0$,}
\eeqs  
 or, in the case $m_0\ge 1$,  it can be estimated, with the use of  \eqref{Pb}, by
\beqs 
\|P_1(x)\|_{p_1}^{\nu_1} \|P_2(x)\|_{p_2}^{\nu_2} \ldots \|P_s(x)\|_{p_s}^{\nu_s} \Big|\frac{\partial P_0(x)}{\partial x_j}\Big|
\le C'|x|^{\sum_{j=1}^s m_j\nu_j}|x|^{m_0-1}=C'|x|^{\beta-1},
\eeqs  
 for some generic constant $C'>0$. Here, and also in calculations below, we use the equivalence between any norm $\|\cdot\|_{p_j}$ and $|\cdot|$.

On the other hand, for each $j=1,\ldots,s$, and $\ell=1,\ldots,n_j$, by using the estimates in \eqref{Pb} again, we have
\begin{align*}
&\nu_j\left(\prod_{\substack{1\le j'\le s,\\ j'\not=j} }
\|P_{j'}(x)\|_{p_{j'}}^{\nu_{j'}} \right)
\|P_j(x)\|_{p_j}^{\nu_j-p_j} |P_{j,\ell}(x)|^{p_j-1} \left|\frac{\partial P_{j,\ell}(x)}{\partial x_i}\right| |P_0(x)|\\
&\le C' \left(\prod_{\substack{1\le j'\le s,\\ j'\not=j} }|x|^{m_{j'}\nu_{j'}} \right)
\|P_j(x)\|_{p_j}^{\nu_j-1} |x|^{m_j-1} |x|^{m_0}\\
&\le C' \left(\prod_{\substack{1\le j'\le m_j,\\ j'\not=j} } |x|^{m_{j'}\nu_{j'}} \right)
|x|^{m_j(\nu_j-1)} |x|^{m_j-1} |x|^{m_0}=C'|x|^{m_0+\sum_{j'=1}^m \nu_{j'} m_{j'}-1}=C'|x|^{\beta-1}.
\end{align*}

Summing up the above estimates after \eqref{FPderiv} and passing $x\to0$, with $\beta>1$, give  
\beqs
\lim_{x\to 0} \frac{\partial F(x)}{\partial x_i}=0=\frac{\partial F(0)}{\partial x_i}.
\eeqs
The last relation comes from the fact $DF(0)=0$ obtained earlier. 
Thus, $\partial F/\partial x_i$ is continuous on $\R^d$, for $i=1,\ldots,d$.
Because $F\in C(\R^d)$ from part \ref{s1}, we obtain $F\in C^1(\R^d)$. 
Consequently, $F$ is locally Lipschitz, and, by combining this with the facts in part \ref{p1}, we conclude $F\in\mathcal X$.
 \end{proof}

In Theorem \ref{thm65}, we usually consider the case $\nu_j/p_j\not\in\N$ for all $j$. Indeed, for an index $j$ with $\nu_j/p_j\in\N$, the corresponding term $\|P_j(x)\|_{p_j}^{\nu_j}$ is a polynomial, and we can combine it with the polynomial $P_0(x)$. 

\begin{example}\label{Eg7}
Regarding condition  \eqref{65b}, it can be met for many forms of $P_j$. For example, if $P_j(x)= (x^T M_1 x)M_0x$ for $x\in\R^d$,
where $M_1$ is a positive definite $d\times d$ matrix, and $M_0$ is an invertible $d\times d$ matrix, then $P_j$ satisfies \eqref{65b}.
\end{example}

\begin{example}\label{Eg8}
Consider  $d=2$ and let 
\beqs
F(x_1,x_2)=(|x_1^3-x_2^3|^{p_1}+|x_1^3+x_2^3|^{p_1})^{\alpha/p_1}\cdot (|x_1 x_2|^{p_2}+|3x_1^2-2x_2^2|^{p_2})^{\beta/p_2}M_0(x_1,x_2),
\eeqs
where $p_1,p_2\ge 2$ are even numbers, $M_0$ is a $\R^2$-valued homogeneous polynomials of degree $m_0\in \Z_+$, and $\alpha,\beta>0$.
Then $F$ is of the form \eqref{heg} with $s=2$, $n_1=n_2=2$, $m_1=3$, $\nu_1=\alpha$, $m_2=2$, $\nu_2=\beta$, and
\beqs
P_1(x)=(x_1^3-x_2^3,x_1^3+x_2^3),\quad  P_2(x)=(x_1 x_2,3x_1^2-2x_2^2).
\eeqs
One can verify that $P_1$ and $P_2$ satisfy \eqref{65b}.
If $m_0+\min\{\alpha,\beta\}>1$, then, thanks to Theorem \ref{thm65}\ref{s3}, $F\in\mathcal X$.
\end{example}

In the remainder of this section, we focus on functions constituted essentially by $x_i^{\gamma_i}$, where $x_i$'s are coordinates of a vector $x\in\R^d$. We will consider more general forms of these power functions, and also combine them with other positively homogeneous functions such as $\|x\|_{p_i}^{\gamma_i}$.

\begin{notation}\label{notation3} 
We will use the following notation for different types of power functions.

\begin{itemize}
 \item Define $\ptset$, a subset of $\R^2$, by $\ptset=(\Z_+\times\{0\})\cup ([0,\infty)\times\{-1,1\})$.
 
 \item For $x\in\R$ and $(\gamma,\tau)\in\ptset$,  denote $\ppp{x}_\tau^\gamma$ as follows
\begin{align}
 &\ppp{x}_0^0=\ppp{x}_1^0=\ppp{x}_{-1}^0=1,\text{ for $\gamma=0$, and } \label{power0}\\
 &\ppp{x}_0^\gamma=x^\gamma,\quad \ppp{x}_1^\gamma=|x|^\gamma,\quad \ppp{x}_{-1}^\gamma=|x|^\gamma\sign(x), \text{ for $\gamma>0$.} \label{power1}
\end{align}
 
 \item For $\gamma=(\gamma_1,\gamma_2,\ldots,\gamma_n)\in\R^n$ and $\tau=(\tau_1,\tau_2,\ldots,\tau_n)\in \R^n$,
denote
\beqs
[\tau,\gamma]=\Big( (\gamma_1,\tau_1),(\gamma_2,\tau_2),\ldots,(\gamma_n,\tau_n)\Big)\in (\R^2)^n.
\eeqs

 \item For vector $x=(x_1,x_2,\ldots,x_n)\in\R^n$, multi-index $\gamma=(\gamma_1,\gamma_2,\ldots,\gamma_n)\in[0,\infty)^n$ and $\tau=(\tau_1,\tau_2,\ldots,\tau_n)\in\{-1,0,1\}^n$ with $[\gamma,\tau]\in\ptset^n$, denote 
\beq\label{vpower}
\ppp{x}_\tau^\gamma =\ppp{x_1}_{\tau_1}^{\gamma_1}\cdot \ppp{x_2}_{\tau_2}^{\gamma_2} \ldots \ppp{x_n}_{\tau_n}^{\gamma_n}.
\eeq
 
 \item For $x\in\R^n$, $p=(p_1,p_2,\ldots,p_n)\in [1,\infty)^n$ and  $\gamma=(\gamma_1,\gamma_2,\ldots,\gamma_n)\in[0,\infty)^n$, denote
\beqs
\|x\|_p^\gamma=\|x\|_{p_1}^{\gamma_1}\cdot \|x\|_{p_2}^{\gamma_2}\ldots\|x\|_{p_n}^{\gamma_n},
\eeqs 
with the convention $\|x\|_{p_i}^0=1$.
\end{itemize}
\end{notation}

The last type of power in \eqref{power1} can be used to re-write the terms like $|x_i|^\alpha x_i$ as $\ppp{x_i}_{-1}^{\alpha+1}$. Also, when some power $\gamma_i$  in \eqref{vpower} is zero, then, thanks to \eqref{power0}, the corresponding term $\ppp{x_i}_{\tau_i}^{\gamma_i}$ is $1$ regardless the value of $x_i$.

\medskip
Let $m\in \N$, $p\in[1,\infty)^m$, $\nu\in [0,\infty)^m$, and $\gamma,\tau\in\R^d$ with $[\gamma,\tau]\in \ptset^d$, and a constant vector $c\in \R^d$.
Then 
\beq \label{mix} 
\text{ the function $x\in\R^d\mapsto \|x\|_p^\nu \ppp{x}_\tau^\gamma\, c$ belongs to $\mathcal H_{|\nu|+|\gamma|}(\R^d)\cap C(\R^d)\cap C^\infty (\R_*^d)$,}
\eeq 
where $|\nu|$ and $|\gamma|$ denote the lengths of multi-indices, see \eqref{cg}. 

In the following presentation, condition $|\nu|=0$ is used to indicate that the term $\|x\|_p^\nu$ is not present in \eqref{mix}. In this case, the values of $m$ and $p$ are irrelevant.

When, in general, the term $\ppp{x}_\tau^\gamma$ is a homogeneous polynomial, or, in particular, $|\gamma|=0$,  the function in \eqref{mix} is reduced to the form \eqref{Fmanyp}, which was already dealt with in Theorem \ref{thm63}.

\begin{theorem}\label{thm69}
Assume that all eigenvectors of matrix $A$ belong to $V=\R^d_*$.
 
\begin{enumerate}[label=\tnum]
 \item Suppose function $F:\R^d\to\R^d$ and number $\beta\in(1,\infty)$ satisfy that $F$ is a finite sum of the functions in \eqref{mix} with  $|\nu|+|\gamma|=\beta$. Then 
 \beq F(0)=0 \text{ and }F\in \mathcal H_\beta(\R^d)\cap C(\R^d)\cap C^\infty (V).
 \eeq 
 Consequently, $F$ belongs to $\mathcal X^0(V)$, and can also play the role of a function $F_k$ in \eqref{Gex} or \eqref{errF3} with $\beta_k=\beta$. 
 
 \item Suppose $F$ is a finite sum of functions in \eqref{mix} with multi-indices $\nu=(\nu_1,\ldots,\nu_m)$ and $\gamma=(\gamma_1,\ldots,\gamma_d)$ satisfying
 \begin{enumerate}[label=\rnum]
  \item\label{l1} $|\nu|+|\gamma|>1$, and
  \item\label{l2} $|\nu|=0$ or ($\forall i=1,\ldots,m:\nu_i\ge 1$), and
  \item\label{l3} $\forall j=1,\ldots,d: \gamma_j=0$ or $\gamma_j\ge 1$.
 \end{enumerate}
 Then $F\in \mathcal X(V)$. 
\end{enumerate}

 \end{theorem}
\begin{proof}
Part (i) clearly comes from property \eqref{mix} and the fact $\beta>1$.
 
Consider part (ii). Thanks to Proposition \ref{prop52}, it suffices to prove (ii) for $F(x)=\|x\|_p^\nu \ppp{x}_\tau^\gamma c$ given as in \eqref{mix} with $p=(p_1,\ldots,p_m)$ and $\tau=(\tau_1,\ldots,\tau_d)$.
By \eqref{mix}, $F\in \mathcal H_\beta(\R^d)\cap C^\infty (V)$, with $\beta=|\nu|+|\gamma|$, which is greater than $1$, thanks to condition \ref{l1}.
Conditions \ref{l2} and \ref{l3} guarantee that the functions $x\in\R^d \mapsto \|x\|_{p_i}^{\nu_i}$, for $i=1,\ldots,m$, and 
$x=(x_1,\ldots,x_d)\in\R^d\mapsto \ppp{x_j}_{\tau_j}^{\gamma_j}$, for $j=1,\ldots,d$, are locally Lipschitz on $\R^d$. Therefore, the function $F$, as a multiplication of these functions and the constant vector $c$, is locally Lipschitz. All together, we have $F\in\mathcal X(V)$.
\end{proof}

\begin{example}\label{Eg10} 
Consider the following system of ODEs in $\R^2$:
\begin{align*}
y_1'+2y_1+y_2&=|y|^{2/3}|y_1|^{1/2}y_2^3,\\
y_2'+y_1+2y_2&= \|y\|_{5/2}^{1/3}y_1|y_2|^{1/4}\sign(y_2). 
\end{align*}
The corresponding matrix $A$ has eigenvalues and bases of the corresponding eigenspaces as follows:
$\lambda_1=1$, basis $\{(-1,1)\}$,
and
$\lambda_2=3$, basis $\{(1,1)\}$.
Then any eigenvector of $A$ belongs to $V=\R^2_*$.
The corresponding function $F$ belongs to $\mathcal X^0(V)$, thanks to Theorem \ref{thm69}(i), and we can apply Theorem \ref{thm57}(ii).
\end{example}

\begin{example}\label{Eg9}
Consider the following system in $\R^2$:
\begin{align*}
y_1'+y_1&=-|y_2|^\alpha y_1,\\
y_2'+y_1+2y_2&= -y_1^2 y_2 , 
\end{align*}
where $\alpha>0$ is not an even integer.
The matrix $A$, its eigenvalues and bases of corresponding  eigenspaces are 
$$
A=\begin{pmatrix}
   1&0\\1&2
  \end{pmatrix},\quad 
\begin{aligned}
\lambda_1&=1,\text{ basis } \{(1,-1)\},\\
\lambda_2&=2, \text{ basis } \{(0,1)\}. 
\end{aligned}
$$

In this case, $F=f+g$, where 
\beq\label{fg}
f(x_1,x_2)=(-|x_2|^\alpha x_1,0)\in\mathcal H_{1+\alpha}(\R^2)\text{ and } 
g(x_1,x_2)=(0,-x_1^2 x_2)\in\mathcal H_3(\R^2).
\eeq

One finds that any eigenvector of $A$ belongs to $V=\R\times \R_*$, and 
\beq\label{fgsmooth} f,g\in C^\infty(V).
\eeq 
Hence, $F\in \mathcal X^0(V)$ and we can apply Theorem \ref{thm57}(ii). 

In case $\alpha\ge 1$, we have $F$ is locally Lipschitz on $\R^2$. This fact, together with \eqref{fg} and \eqref{fgsmooth}, implies that  $F\in \mathcal X(V)$ and we can apply Theorem \ref{thm57}(i).
\end{example}

\begin{example}\label{Eg11} There are many other situations, especially in multi-dimensional spaces higher than $\R^2$. We present one example here. Let $d=3$, and assume $3\times 3$ matrix $A$ has the following eigenvalues and bases of the corresponding eigenspaces 
$$\lambda_1=\lambda_2=1, \text{ basis }  \{\xi_1=(1,0,1), \xi_2=(0,1,0)\},
\text{ and }\lambda_3=2, \text{ basis } \{\xi_3=(1,1,-1)\}.$$

Let $F(x)=(x_1^2+x_2^2)^{1/3} \cdot (x_2^6+x_3^6)^{1/5}P(x)$, where $P$ is a polynomial vector field on $\R^3$ of degree $m_0\in\N$ without the constant term, i.e., $P(0)=0$. 

Suppose $\xi$ is an eigenvector of $A$. Then
$\xi=c_1\xi_1+c_2\xi_2 $ for $c_1^2+c_2^2>0$, or $\xi=c_3\xi_3$ for $c_3\ne 0$.
One can verify that
\begin{align*}
\xi\in V&=\{(x_1,x_2,x_3):x_2\ne 0 \text{ or } x_1x_3\ne 0\}\\
&=(\R\times \R_*\times \R)\cup (\R_*\times\R\times \R_*)=(\R^2_0\times \R)\cap (\R\times\R^2_0). 
\end{align*}

Note that $F\in \mathcal H_{\beta}(\R^3)\cap C^\infty(V)$ with $\beta=2/3+6/5+m_0$, and, thanks to Theorem \ref{thm65}\ref{p3}, $F\in C^1(\R^3)$.
Then $F\in \mathcal X(V)$ and, according to Theorem \ref{thm57}(i), we can apply Theorem \ref{thm3}  to obtain an infinite series asymptotic expansion for any non-trivial, decaying solution $y(t)$ of \eqref{vtoy}.
 \end{example}
 
 \begin{example}[by A.~D.~Bruno] Below is a specific case when a solution has a similar, but different, asymptotic expansion. The system
 \beq\label{bruno}
\begin{aligned}
y_1' +y_1&=0\\
y_2' +3y_2& =\frac 3 2 y_1^2y_2^{1/3} 
\end{aligned}
\eeq 
 has a solution 
$(y_1,y_2)=(e^{-t},t^{3/2} e^{-3t})$
which, thanks to  the term $t^{3/2}$, does not have an expansion \eqref{ex1}. 

We can examine system \eqref{bruno} and see that it does not satisfy the conditions in Theorems \ref{thm3}, \ref{thm51} and \ref{thm57}. Indeed, we always require that each positively homogeneous function $F_k$  in approximations \eqref{Gex}, \eqref{Gef}, \eqref{errF3} of $F$ is infinitely differentiable in some neighborhood of any eigenvector of the matrix $A$, see Assumption \ref{assumpG}(ii) and Definition \ref{notation2}. In the current example, 
$$A=\begin{pmatrix}
   1&0\\0&3
  \end{pmatrix} \text{ and } F(x_1,x_2)=(0,\frac32x_1^2x_2^{1/3}).
$$
Clearly, $\xi=(1,0)$ is an eigenvector of $A$ but $F_1=F$, with degree $\beta_1=2+1/3$, is not a $C^\infty$-function in any neighborhood of $\xi$. Thus, our results (Theorems \ref{thm3}, \ref{thm51} and \ref{thm57}) cannot be applied to system \eqref{bruno}.
 \end{example}
 
\begin{remark}\label{discuss}
In case $F$ is analytic, Lyapunov's First Method yields that a decaying solution solution $y(t)$ of \eqref{vtoy} equals a series $\sum_{n=1}^\infty q_n(t)e^{-\mu_n t}$ for sufficiently large $t$, where $q_n(t)$'s are some polynomials. See e.g. \cite[Chapter I, \S4]{BibikovBook} where the proof is based on the Poincar\'e--Dulac normal form. 
Bruno investigates a much larger class of equations of differential sums, which are not necessarily of a dissipative type like ours. He develops the theory of power geometry and \emph{finds} solutions that have certain forms of asymptotic expansions. Specific algorithms are developed to calculate those asymptotic expansions. See \cite{BrunoBook1989,BrunoBook2000,Bruno2004,Bruno2008c,Bruno2012,Bruno2018} and references there in. 
His equations can have complex values, and the nonlinearity is comprised of power functions.
His method and results are totally different from ours. For example, he does not obtain the  particular expansion \eqref{FSx}.
Also, we obtain the asymptotic expansions for \emph{any given} non-trivial, decaying solutions, and our nonlinearity, in case of real-valued functions, can contain more general terms such as in \eqref{heg} and \eqref{mix}.
\end{remark}
 
\def\cprime{$'$}

\end{document}